\documentclass{amsart}

\usepackage[table]{xcolor}

\usepackage{hyperref}

\usepackage{enumerate}
\usepackage{upgreek}
\usepackage{bm}
\usepackage{pstool}
\usepackage{thmtools}
\usepackage{thm-restate}

\usepackage{url}

\newcommand{\proofof}[1]{\underline{\emph{Proof of} #1}}
\newenvironment{proofofpart}[1]%
	{\begin{proof}[\underline{Proof of {#1}}\textnormal{:}]}%
	{\end{proof}}

\definecolor{darkslateblue}{rgb}{0.28, 0.24, 0.55}
\definecolor{antiquefuchsia}{rgb}{0.57, 0.36, 0.51}
\definecolor{desertsand}{rgb}{0.93, 0.79, 0.69}

\DeclareFontFamily{U}{MnSymbolC}{}
\DeclareSymbolFont{MnSyC}{U}{MnSymbolC}{m}{n}
\DeclareMathSymbol{\diamonddot}{\mathbin}{MnSyC}{"7E}
\DeclareFontShape{U}{MnSymbolC}{m}{n}{
    <-6>  MnSymbolC5
   <6-7>  MnSymbolC6
   <7-8>  MnSymbolC7
   <8-9>  MnSymbolC8
   <9-10> MnSymbolC9
  <10-12> MnSymbolC10
  <12->   MnSymbolC12}{}

\newcommand{\ScalProd}{\bm{\cdot}}

\newcommand{\MinkProd}{\diamonddot}

\usepackage{amssymb}
\usepackage{amsthm}
\usepackage{tikz}

\theoremstyle{definition}

\newtheorem{theorem}{Theorem}[subsection]

\newtheorem{definition}[theorem]{Definition}
\newtheorem{lemma}[theorem]{Lemma}

\newtheorem{remark}[theorem]{Remark}
\newtheorem{corollary}[theorem]{Corollary}

\newtheorem{problem}{Problem}

\newcommand{\ca}[1]{\mathsf{Conc}\,#1\,}

\newcommand{\defeq}{\ \stackrel{\textsf{\tiny def}}{=} \ }
\newcommand{\defiff}{\stackrel{\textsf{\tiny def}}{\iff}}
\newcommand{\ie}{i.e., }
\newcommand{\eg}{e.g., }

\newcommand{\laa}{\langle}
\newcommand{\raa}{\rangle}
\newcommand{\la}{\langle}
\newcommand{\ra}{\rangle}

\newcommand{\PoiSim}{\mathsf{PoiSim}}
\newcommand{\Eucl}{\mathsf{Eucl}}
\newcommand{\EuclSim}{\mathsf{EuclSim}}

\newcommand{\TrivSim}{\mathsf{TrivEuclSim}}
\newcommand{\GenTriv}{\mathsf{TrivGalSim}}

\newcommand{\GalSim}{\mathsf{GalSim}}
\newcommand{\GenGal}{\mathsf{GalSim}}

\newcommand{\minkkerning}{\hspace{-1pt}}

\newcommand{\minkst}{\mathcal{M}\minkkerning\mathit{ink}}

\newcommand{\newtst}{\mathcal{N}\!\mathit{ewt}}

\newcommand{\galst}{\mathcal{G}\mathit{al}}

\newcommand{\euclg}{\mathcal{E}\!\mathit{ucl}}

\newcommand{\cng}[1]{\bm{{ \cong}}_{#1}}

\newcommand{\llcon}{\bm{\uplambda}}
\newcommand{\abstime}{{\mathsf{S}}}
\newcommand{\rest}{\mathsf{Rest}}

\newcommand{\simult}{\mathbf{S}}
\newcommand{\taxis}{\mathbf{T}}

\newcommand{\Nodel}{{\mathfrak{N}}}
\newcommand{\Model}{{\mathfrak{M}}}

\newcommand{\origin}{\vec{\textsf{\small o}}}
\newcommand{\unit}{\vec{\textsc e}}

\newcommand{\aut}[1]{\mathsf{Aut}\; #1}

\newcommand{\AffAut}[1]{\mathsf{AffAut}\,#1}
\newcommand{\fv}[1]{{#1}}
\newcommand{\OAff}{\mathcal{OA}\mathit{ff}}

\newcommand{\relst}{\mathcal{R}\mathit{el}}

\newcommand{\lclst}{\uplambda\mathcal{C}\mathit{lass}}
\newcommand{\Lclst}{\mathcal{LC}\mathit{lass}}

\newcommand{\field}{\mathfrak{F}}

\newcommand{\rel}{\mathsf{R}}

\newcommand{\F}{F}
\newcommand{\geometry}{\mathcal{G}}

\newcommand{\Bw}{{\mathsf{Bw}}}

\newcommand{\sqEd}[2]{\mathsf{d}^2(#1,#2\,)}
\newcommand{\sqMd}[2]{\mathsf{d}_{\mu}^2\left(#1,#2\,\right)}

\newcommand{\sqDist}[2]{\mathsf{d}_{\otimes}^2\left(#1,#2\,\right)}
\newcommand{\sqLen}[1]{\|#1\,\|_{\otimes}^{2}}

\title{Definable coordinate geometries over fields, part 2: applications} 
\author{Judit Madar\'asz}
\address{Judit Madar\'asz, HUN-REN Alfr\'ed R\'enyi Institute of Mathematics, Budapest, Hungary}
\email{madarasz.judit@renyi.hu}

\author{Mike Stannett}
\address{Mike Stannett, School of Computer Science, The University of Sheffield, Sheffield, UK}
\email{m.stannett@sheffield.ac.uk}

\author{Gergely Sz\'ekely}
\address{Gergely Sz\'ekely, HUN-REN Alfr\'ed R\'enyi Institute of Mathematics, Budapest, Hungary  \& University of Public Service, Budapest, Hungary.}
\email{szekely.gergely@renyi.hu}

\date{\today}

\thanks{This research is supported by the Hungarian National
  Research, Development and Innovation Office (NKFIH), grants
  no.\ FK-134732 and TKP2021-NVA-16.}

\begin{document}
\maketitle

\begin{abstract}
In Part 1 of this study we showed, for a wide range of geometries,
that the relationships between their concept-sets are fully determined
by those between their (affine) automorphism groups. In this
(self-contained) part, we show how this result can be applied to
quickly determine relationships and differences between various
geometries and spacetimes, including ordered affine, Euclidean,
Galilean, Newtonian, Late Classical, Relativistic and Minkowski
spacetimes (we first define these spacetimes and geometries using a
Tarskian first-order language centred on the ternary relation $\Bw$ of
betweenness). We conclude with a selection of open problems related to
the existence of certain intermediate geometries.
\end{abstract}

Keywords: \keywords{Definability; Concepts; Coordinate geometries;
  Minkowski spacetime; Newtonian spacetime; Galilean spacetime;
  Automorphisms}

\section{introduction}
In Part 1 \cite{part1}, we introduced the notion of finitely
field-definable coordinate geometries over fields and ordered fields,
and proved that a relation is explicitly definable in such a geometry
exactly if it is field-definable and respected by the (affine)
automorphisms of the geometry \cite[Thm.5.1.2]{part1}. This result and
its associated Corollary 5.1.5 (recalled here as
Corollary~\ref{re-cor-erlangen} below) yield a simple but powerful
strategy for comparing the conceptual content of different geometries:
\emph{If we wish to compare the sets of concepts of finitely
field-definable coordinate geometries over the same field and of the
same dimension, it is enough to compare their affine automorphism
groups.}

In this (self-contained) paper, we demonstrate the power of this
approach by showing the connection between the concept-sets of various
classical geometries, including Galilean, Newtonian and Relativistic
spacetimes and Euclidean geometry, as illustrated in
Figures~\ref{fig:hasse} and \ref{tartalmazas-fig}. We define these
spacetimes and geometries using a Tarskian first-order language
centred around the ternary relation $\Bw$ of betweenness. The language
we use to model Galilean spacetime is the same as the one used by
\cite{ketland23}, who gives a second-order axiomatization for Galilean
spacetime using the basic notions of betweenness ($\Bw$), simultaneity
($\abstime$), and congruence on simultaneity ($\cng{\abstime}$).

In a spirit similar to ours, but using a metric approach to define
Galilean, Newtonian and Minkowski spacetimes instead of a Tarskian
first-order language, \cite{Barrett15} compares the ``amount of
structure'' in these models.\footnote{The results in \cite{Barrett15}
seemingly contradict our results because there, the automorphism group
of Newtonian spacetime is a subgroup of that of Minkowski spacetime,
whereas in our work it is not. This contradiction is only apparent; it
comes from the fact that the metric approach gives more rigid models
of these spacetimes (it fixes the units of measurements) while our
Tarskian approach gives unit-free models: certain scalings,
intuitively corresponding to changing the units of measurement, are
automorphisms of the spacetimes considered here but not those of \cite{Barrett15}).}  Similar
investigations of the lattices of concept-sets of various structures
can be found in \cite{SSU14,MuSem20,SemSop21,SemSop22,SemSop24}.

The work presented here is part of an ongoing project aiming to
understand the algebras of concepts of concrete mathematical
structures and the conceptual
  distances between them \cite{KSzLF20,KSz21,KSz25}. From a more
general point of view, this research is part of the wider
Andr\'eka--N\'emeti school's project aiming to provide
a logic-based foundation and understanding of relativity theories in
the spirit of the Vienna Circle and Tarski's initiative,
\textit{Logic, Methodology and Philosophy of Science}
\cite{AMNNSz11,Friend15,FF21}.

\section{Recalling some notations and results}

The notation and definitions used in this paper follow those of
\cite{part1}, to which the reader is referred for more details. We use
standard notions of model theory and set theory, and we assume that
the reader is familiar with basic algebraic structures like groups,
fields and vector spaces. By a \emph{model}, we mean a first-order
structure in the sense of model theory.
 
Let $M$ be a nonempty set and let $f\colon M \to M$ be a map.  We say that $f$ \emph{respects} $n$-ary relation $\rel \subseteq M^n$ if{}f $(a_1,\dots, a_n)\in \rel \iff (f(a_1),\dots, f(a_n))\in\rel$.

When we talk about \emph{concepts} of model $\Model$, we mean the
relations definable on its universe $M$. We say that an $n$-ary
relation $\rel$ on $M$ is \emph{definable} in $\Model$ if{}f there is
a first-order formula $\varphi(\fv{v}_1,\ldots,\fv{v}_n)$ in the
language of $\Model$ with free variables among
$\fv{v}_1,\ldots,\fv{v}_n$ that defines it; \ie for every
$(a_1,\ldots,a_n)\in M^n$, we have
\begin{equation}\label{eq:def-def}
(a_1,\ldots,a_n) \in \rel \quad\Longleftrightarrow\quad
\Model\models\varphi[a_1,\ldots,a_n].
\end{equation}
The set of all concepts of $\Model$ (\ie the set of all relations
definable in $\Model$) is denoted by $\ca{\Model}$.  Models $\Model$
and $\Nodel$ are \emph{definitionally equivalent} if{}f they have the
same concepts, \ie $\ca{\Model}=\ca{\Nodel}$. This notion of
definitional equivalence is just a reformulation of the usual one that
fits best for this paper, see \eg \cite[p.51]{HMT71} or
\cite[p.453]{Monk2000} for the standard definition.

Let $\field=\la \F,+,\cdot,0,1,\leq\ra$ be an ordered field.%
\footnote{That $\la\F,+,\cdot,0,1,\leq\ra$ is an ordered field means
that $\la\F,+,\cdot,0,1\ra$ is a field which is totally ordered by
$\leq$, and the following two properties holds for all ${x},{y},{z}\in
\F$: (i) ${x}+{z}\leq {y}+{z}$ if ${x}\leq {y}$, and (ii) $0\leq
  {x}{y}$ if $0\leq {x}$ and $0\leq {y}$.}  Let us fix a natural
number $d\ge 2$ for the dimension.  We will denote points in $F^d$
($d$-tuples over $F$) using arrows and angle-brackets. Given a
$d$-tuple $\vec{p}\in F^d$, we denote it's $i$'th component by $p_i$,
\ie $\vec{p} = \la p_1, \dots, p_d \ra$ by this convention.  We write
$\origin$ to denote the \emph{origin} $\laa 0, \dots, 0 \raa \in F^d$,
and $\unit_1,\ldots,\unit_d$ to denote the \emph{unit vectors} $\laa
1,0,\ldots,0\raa, \ldots, \laa 0,\ldots,0,1\raa \in F^d$,
respectively.  The ternary relation $\Bw$ of \emph{betweenness} on
$\F^d$ is defined for every $\vec{p},\vec{q},\vec{r}\in\F^d$ as
\[
\Bw(\vec{p},\vec{q},\vec{r}\,)\defiff 
\vec{q}=\vec{p}+\lambda(\vec{r}-\vec{p}\,)\ \text{ for some }\ \lambda\in [0,1] \subseteq \F.
\]
A model $\geometry$ is called a ($d$-dimensional) \emph{coordinate
{geometry} over ordered field $\field$} if{}f the following conditions
are satisfied:
\begin{itemize}
\item the universe of $\geometry$ is $\F^d$ (called the set of points);
\item $\geometry$ contains no functions or constants;
\item
the ternary relation $\Bw$ of betweenness on points is definable in
$\geometry$.% if $\field$ is an ordered field.
\end{itemize}

There is a natural correspondence between $n$-tuples of points in
$F^d$ ($n$-tuples of ($d$-tuples over $F$)) and $(dn)$-tuples over
$F$.  We say that an $n$-ary relation $\rel\subseteq
\left(F^d\right)^n$ is \emph{definable over $\field$} if{}f the
corresponding relation $\widehat{\rel}\subseteq F^{dn}$ is definable
in $\field$. For example, ternary relation $\Bw$ on points is
definable over $\field$ in this sense.

We call coordinate geometry $\geometry$ \emph{field-definable} if{}f all the relations of $\geometry$ are definable over $\field$; and we call $\geometry$ \emph{finitely field-definable} if{}f it is definable and it contains only finitely many relations.

Model $\la \F^d,\Bw\ra$ is called $d$-dimensional \emph{ordered affine geometry} over ordered field $\field$. This definition of ordered affine geometry is exactly what is called affine Cartesian space in \cite[Def.1.5.]{SzT79}.

A map $A\colon\F^d\rightarrow\F^d$ is called an \emph{affine
transformation} if{}f it is the composition of a linear bijection and
a translation.  By \emph{affine automorphisms} of coordinate geometry
$\geometry$ over $\field$, we mean affine transformations that are
also automorphisms of $\geometry$. We denote the set of all
automorphisms of $\geometry$ by $\aut{\geometry}$ and the set of all
affine automorphisms by $\AffAut{\geometry}$.

The following theorem and corollary from \cite{part1} will be used
here:
\begin{theorem}
\label{re-thm-erlangen}
Assume that  $\geometry$ and $\geometry'$ are
finitely field-definable coordinate geometries over
$\field$. Then:
\begin{enumerate}[(i)]
\item
$\ca{\geometry}\subseteq\ca{\geometry'}\ \Longleftrightarrow\
\aut{\geometry}\supseteq\aut{\geometry'}$.
\item
\label{Er2}
$\ca{\geometry}\subseteq\ca{\geometry'}\ \Longleftrightarrow\
 \AffAut{\geometry}\supseteq
\AffAut{\geometry'}$. %\hfill $\Box$
\end{enumerate}
\end{theorem}

\begin{corollary}
\label{re-cor-erlangen}
Assume that $\geometry$ and $\geometry'$ are finitely field-definable
coordinate geometries over $\field$. Then:
\begin{enumerate}[(i)]
\item  $\ca{\geometry}=\ca{\geometry'}\ \Longleftrightarrow\
\aut{\geometry}=\aut{\geometry'}$.
\item 
\label{ErCor2}
$\ca{\geometry}=\ca{\geometry'}\ \Longleftrightarrow\
\AffAut{\geometry}=\AffAut{\geometry'}$.
\item  $\ca{\geometry}\subsetneq\ca{\geometry'}\ \Longleftrightarrow\
\aut{\geometry}\supsetneq\aut{\geometry'}$.
\item $\ca{\geometry}\subsetneq\ca{\geometry'}\ \Longleftrightarrow\
\AffAut{\geometry}\supsetneq\AffAut{\geometry'}$. \hfill $\Box$
\end{enumerate}
\end{corollary}

There is an ongoing research on various criterions to determine when a
mathematical object has more structure than another one, see \eg
\cite{North09,SwansonHalvorson12,Barrett14,Wilhelm21,Barrett22,BMW23}. One
of the criterions in that research is (SYM*), which is formulated in
\cite{SwansonHalvorson12} as: ``If $Aut(X)$ properly contains
$Aut(Y)$, then $X$ has less structure than $Y$.'' By
Corollary~\ref{re-cor-erlangen}, (SYM*) works perfectly for finitely
field-definable coordinate geometries if `structure' is understood as
the set of explicitly definable relations.

\section{Examples and Applications}
\label{sec:examples-and-applications}

The power of this approach shows itself, for example, in our proof of
Theorem~\ref{autgroups}\eqref{Beq3} below that affine automorphisms
respect lightlike relatedness if and only if they respect Minkowski
equidistance, which implies that these two concepts are interdefinable
(using betweenness). Finding the explicit formula that defines
Minkowski equidistance from lightlike relatedness and betweenness is
not trivial, whereas comparing the relevant affine automorphism groups
is relatively straightforward.

In the remainder of this section, we show how this strategy can be
used to demonstrate the conceptual relationships between various
historically significant geometries (see Figure~\ref{fig:hasse}),
thereby establishing which of these geometries are and are not
definitionally equivalent to one another. The main geometries
considered here are Ordered Affine ($\OAff$) and Euclidean ($\euclg$)
geometries, and Relativistic ($\relst$), Minkowski ($\minkst$),
Galilean ($\galst$), Newtonian ($\newtst$) and Late Classical
($\Lclst$) spacetimes.

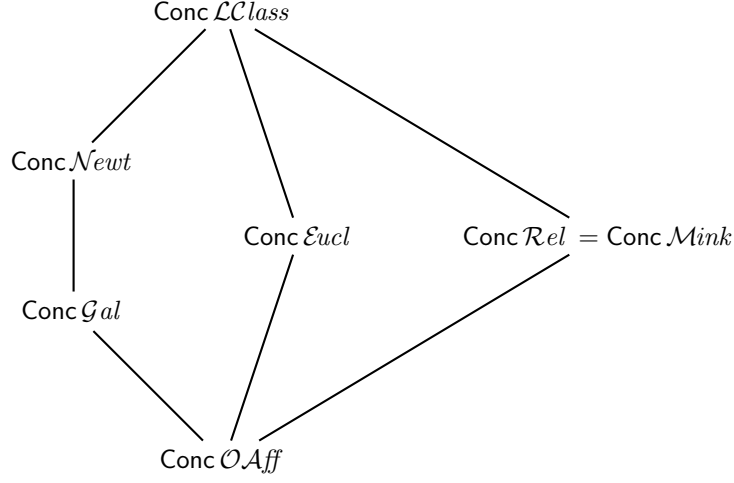
\begin{figure}
  \begin{tikzpicture}[scale=1]
    \tikzstyle{edge}=[thick]
    \node (LC) at (0,6){$\ca{\Lclst}$};
    %\node (R) at (4,3){{$\ca{\minkst}$}};
    \node (R) at (5,3){$\ca{\relst}=\ca{\minkst}$};
    \node (E) at (1,3){$\ca{\euclg}$};
    \node (G) at (-2,2){$\ca{\galst}$};
    \node (N) at (-2,4){$\ca{\newtst}$};
    \node (A) at (0,0) {$\ca{\OAff}$};

    \draw[edge] (A) to (G) to (N) to (LC);
    \draw[edge] (A) to (E) to (LC);
    \draw[edge] (A) to (R) to (LC);
  \end{tikzpicture}
  \caption{Hasse diagram showing how the concept-sets associated with
  various historically significant geometries are related to one another by subset inclusion.}
  \label{fig:hasse}
\end{figure}

\subsection{Assumptions}
Throughout, we restrict attention to $d$-dimensional coordinate
geometries over an ordered field $\field=\la \F,+,\cdot,0,1,\leq \ra$
where, as always, we assume that $d \geq 2$.

\begin{definition}[Products, lengths and distances]
A function $\otimes\colon F^d \times F^d \to F$ 
will be called a \emph{product} on $F^d$ if{f} 
it is both symmetric and bilinear.
If $\otimes$ is a product on $F^d$, its
corresponding \emph{squared length} 
function is the function 
$\sqLen{\vec{p}} \defeq \vec{p}\otimes\vec{p}$ 
and its corresponding 
\emph{squared distance} function is the function $\sqDist{\vec{p}}{\vec{q}} \defeq \sqLen{\vec{p}-\vec{q}}$, for all $\vec{p}, \vec{q} \in F^d$. 
\end{definition}

In particular, the Euclidean
 \emph{scalar product} $\ScalProd$  of vectors 
$\vec{p}, \vec{q} \in	\F^d$ is defined, as usual, by
\[
\vec{p}\ScalProd\vec{q} \defeq p_1q_1+\ldots+p_dq_d,
\]
and their \emph{Minkowski  product} ($\MinkProd$) 
by
\[ 
\vec{p}\MinkProd \vec{q} \defeq p_1q_1-p_2q_2-\ldots-p_dq_d.
\]

The squared distance functions corresponding to the Euclidean and Minkowski products 
will be denoted by $\mathsf{d}^2$ and $\mathsf{d}_\mu^2$, respectively.

\subsection{Unit-free spacetimes}

Thinking of $F^d$ as the set of spacetime points, we regard the first
component, $p_1$, of any vector $\vec{p} = \langle p_1,p_2,\dots,p_d
\rangle$ as its time component, and the rest, $(p_2,\dots,p_d)$ as its
spatial component. This is reflected in the following terminology.

On the set of points $\F^d$, the binary relations $\abstime$ of
\emph{absolute simultaneity}, $\rest$ of \emph{being at rest} and
$\llcon$ of \emph{lightlike relatedness} are defined for $\vec{p},
\vec{q} \in \F^d$ by:
\begin{align*}
\vec{p}\,\abstime\,\vec{q} &\ \defiff\  p_1=q_1,\\
\vec{p}\,\rest\,\vec{q} &\ \defiff\  ( p_2,\ldots,p_d) = (q_2,\ldots,q_d), \\
\vec{p}\,\llcon\,\vec{q} & \ \defiff\  (p_1-q_1)^2=(p_2-q_2)^2+\ldots+(p_d-q_d)^2.
\end{align*}
We note that
\begin{equation}
\label{llcon-alt}
    \vec{p}\,\llcon\,\vec{q} \quad\Longleftrightarrow\quad  \sqMd{\vec{p}}{\vec{q}}=0 .
\end{equation}

On the set of points
$\F^d$,
the $4$-ary relations $\cng{}$ of \emph{(Euclidean) congruence},
$\cng{\mu}$ of \emph{Minkowski congruence}, and $\cng{\abstime}$ of
\emph{congruence on simultaneity} are defined for
$\vec{p},\vec{q},\vec{r},\vec{s}\in\F^d$ by:
\begin{align*}
(\vec{p},\vec{q}\,)\,\cng{}\,(\vec{r},\vec{s}\,) & \defiff  
\sqEd{\vec{p}}{\vec{q}}=\sqEd{\vec{r}}{\vec{s}}, \\
(\vec{p},\vec{q}\,)\,\cng{\mu}\,(\vec{r},\vec{s}\,) & \defiff  
\sqMd{\vec{p}}{\vec{q}}=\sqMd{\vec{r}}{\vec{s}}, \\
(\vec{p},\vec{q}\,)\,\cng{\abstime}\,(\vec{r},\vec{s}\,)
& \defiff  (\vec{p},\vec{q}\,)\,\cng{}\,(\vec{r},\vec{s}\,)\ \land\
\vec{p}\,\abstime\,\vec{q}\ \land\  \vec{r}\,\abstime\,\vec{s}
\end{align*}                                         

When $(\vec{p},\vec{q}\,)\,\cng{}\,(\vec{r},\vec{s}\,)$ holds, we say
that $(\vec{p},\vec{q}\,)$ and $(\vec{r},\vec{s}\,)$ are
\emph{congruent segments}, and when
$(\vec{p},\vec{q}\,)\,\cng{\abstime}\,(\vec{r},\vec{s}\,)$ holds, we
say that they are \emph{spatially congruent}.

Note that the relations $\llcon$, $\abstime$, $\rest$,
$\cng{}$, $\cng{\mu}$ and $\cng{\abstime}$ are all definable over 
$\field$.

We also define a ternary relation $\bm{\updelta}$ 
on $\F^d$ by
\[
\bm{\updelta}(\vec{p},\vec{q},\vec{r}\,)\ \defiff\
\vec{p}\,\abstime\,\vec{q}\ \land\ \vec{p}\,\llcon\,\vec{r}.
\]
Notice that $\la \F^d,\abstime,\llcon\ra$ and
$\la \F^d,\bm{\updelta}\ra$ are definitionally equivalent
\ie 
\begin{equation}
\label{deltaeq}
\ca{\la \F^d,\abstime,\llcon\ra} =
\ca{\la \F^d,\bm{\updelta}\ra}
\end{equation}
because $\bm{\updelta}$ is defined in terms
of $\abstime$ and $\llcon$, and it is not difficult to prove 
that
both $\abstime$ and $\llcon$ can be defined in terms of
$\bm{\updelta}$ by: 
$
\vec{p}\,\abstime\,\vec{q}\; \Leftrightarrow\;  \exists \vec{r}\,
\bm{\updelta}(\vec{p},\vec{q},\vec{r}\,)$ and
$\vec{p}\,\llcon\,\vec{r}\; \Leftrightarrow\;  
\exists \vec{q}\, \bm{\updelta}(\vec{p},\vec{q},\vec{r}\,)$.

All spacetimes and geometries considered in this paper are
unit-free.\footnote{By ``unit-free'' we mean that there is no specific
units fixed for measuring spatial or temporal distances.} We now
recall the definition of ($d$-dimensional) ordered affine geometry,
and introduce Euclidean geometry and, respectively, Relativistic,
Minkowski, Galilean, Newtonian and Late Classical\footnote{Late
classical spacetime is a geometry corresponding to the understanding
of classical kinematics during the time of the search for the
luminiferous ether before Einstein's introduction of Special
Relativity in 1905, cf.\ also \cite{diss,ClassRelKin} for an
alternative approach to this kinematics.}  spacetimes over $\field$:
\begin{align*}
\OAff   &\ \defeq \ \la \F^d, \Bw \ra, \\
\euclg  &\ \defeq \ \la \F^d,\cng{} ,\Bw\ra,\\
\relst  &\ \defeq \ \la \F^d,\llcon,\Bw \ra,\\
\minkst &\ \defeq \ \la \F^d,\cng{\mu},\Bw\ra,\\ 
\galst  &\ \defeq \ \la \F^d,\cng{\abstime},\Bw\ra,\\
\newtst &\ \defeq \ \la \galst,\rest \ra,\\
 \Lclst &\ \defeq \ \la \galst, \llcon \ra
 \end{align*}
and note that these are all finitely field-definable coordinate
geometries over $\field$.

\begin{remark}
If
 $\field$ is a \emph{Euclidean}
 field 
(\ie one in which positive elements have 
square roots), then
 $\Bw$ is definable in the
$\Bw$-free reduct $\la \F^d,\cng{}\ra$
of $\euclg :=\la\F^d,\cng{},\Bw\ra$; if $d\geq 3$
and $\field$ is Euclidean, then
$\Bw$ is also definable in the $\Bw$-free reducts
of $\relst$, $\minkst$ and $\Lclst$.
On the other hand,
 $\Bw$ is not definable in the $\Bw$-free
reducts of $\galst$ and $\newtst$ for any $d \geq 2$.
\end{remark}

Given the close historical relationship between Relativistic spacetime
and Minkowski spacetime, one would naturally expect to find, as shown
in Fig.~\ref{tartalmazas-fig}, that $\minkst$ and $\relst$ are
conceptually identical. Both this relationship and the other subset
relationships shown in Fig.~\ref{tartalmazas-fig} are summarised in
the following theorem.

\begin{theorem}
\label{leiras}
The subset relations between
 $\ca{\relst}$, $\ca{\minkst}$, 
$\ca{\euclg}$, $\ca{\galst}$, $\ca{\newtst}$,
$\ca{\Lclst}$  and $\ca{\OAff}$  are as  shown  
in Fig.~\ref{tartalmazas-fig}. 
In particular, the concepts  $\abstime$, $\rest$,
$\llcon$, $\cng{}$, $\cng{\abstime}$,  $\cng{\mu}$ and 
$\bm{\updelta}$
allow us to distinguish between different geometries, as shown
explicitly in Table~\ref{tab:concept-locations}.  
In more detail:
\begin{enumerate}[(i)]
\item
\label{Le1}
$\ca{\relst}=\ca{\minkst}$, \ie $\relst$ and $\minkst$ are
definitionally equivalent.
\item  
\label{Le2}
$\ca{\galst}\subseteq\ca{\newtst}$.
\item 
\label{Le3}
$\ca{\OAff}\subseteq 
\left(\ca{\euclg}\cap\ca{\relst}\cap\ca{\galst}\right)$.
\item
\label{Le4}
$\left(\ca{\euclg}\cup\ca{\relst}\cup
\ca{\newtst}\right)\subseteq
\ca{\Lclst}$.
\item 
\label{Le5}
Table~\ref{tab:concept-locations} correctly identifies which of the 
cited geometries do and do not contain the concepts $\abstime$, $\rest$, 
$\llcon$, $\cng{}$, $\cng{\abstime}$,  $\cng{\mu}$ and $\bm{\updelta}$.
\hfill $\Box$
\end{enumerate}
\end{theorem}

\begin{figure}
\begin{tikzpicture}[scale=0.57]
    \tikzstyle{vonal}=[ultra thick]
    \tikzstyle{bogyo}=[fill, circle, inner sep=1.5]

    \draw[vonal,dashed] (1,.4) arc [start angle=0, end angle=360, x radius=0.9cm, y radius=0.9cm]  node [pos=.65, below, rotate=-35] {$\scriptstyle\ca{\OAff}$};
    \draw[vonal] (10,0) arc [start angle=0, end angle=360, x radius=10cm, y
radius=6.5cm] node [pos=.2,above right,rotate=-13] {$\ca{\Lclst}$};

    \draw[vonal,green!31!black] (9,0) arc [start angle=0, end angle=360, x radius=6cm, y
radius=3.2cm] node [pos=.2,above,black,rotate=-10] {$\ca{\newtst}$};

    \draw[vonal,green!42!black] (6,0) arc [start angle=0, end angle=360, x radius=4cm, y radius=2cm] node [pos=.2,above,black,rotate=-5] {$\ca{\galst}$};

    \draw[vonal,blue!63!black,rotate=150] (6,0) arc [start angle=0, end angle=360, x radius=4cm, y radius=2cm] node [pos=.95,above,black,rotate=20] {$\ca{\euclg}$};

    \draw[rotate=-155,vonal,red!63!black] (6,0) arc [start angle=0, end angle=360, x radius=4cm, y radius=2cm] node [pos=.08,below,black,rotate=-15] {$\ca{\minkst}=\ca{\relst}$};

    \node[bogyo, label={60:$\cng{}$} ] (cg) at (-3.8,2){};
    \node[bogyo, label={60:$\bm{\updelta}$} ] (d) at (-7,1){};
    \node[bogyo, label={60:$\abstime$} ] (S) at (3,0.7){};
    \node[bogyo, label={60:$\Bw$} ] (bw) at (-0.2,0.1){};
    \node[bogyo, label={60:$\rest$} ] (R) at (7,0){};
    \node[bogyo, label={60:$\cng{\abstime}$} ] (cgS) at (3.3,-1){};
    \node[bogyo, label={60:$\llcon$} ] (l) at (-4,-1.5){};
    \node[bogyo, label={60:$\cng{\mu}$} ] (m) at (-3,-2.7){};
  \end{tikzpicture}

\caption{Venn diagram showing how the concept-sets associated with various
geometries are related to one another. The diagram shows which geometries
do and do not contain various key concepts; see 
Table~\ref{tab:concept-locations}.}
\label{tartalmazas-fig} 
\end{figure}
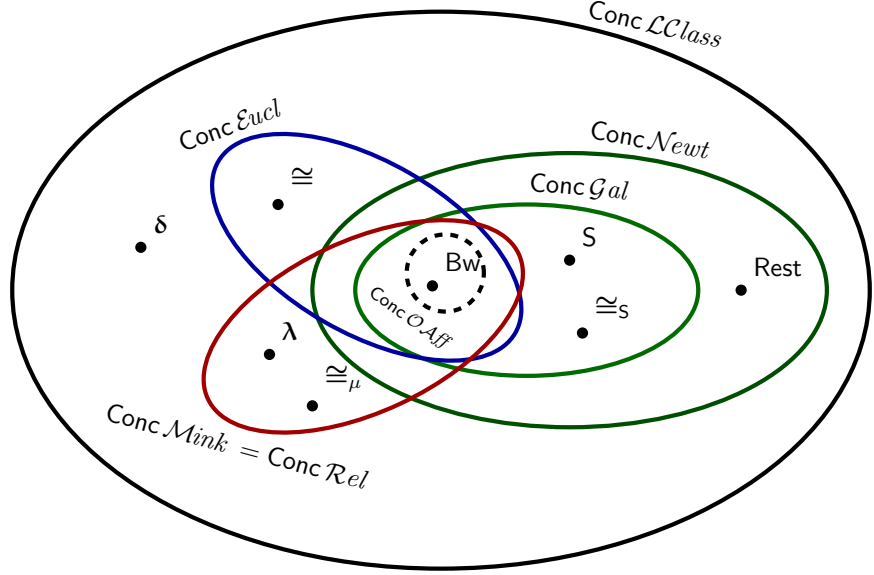

\begin{table}[ht]
\caption{The geometries shown in Fig.~\ref{tartalmazas-fig} which
 do and do not contain
the concepts $\abstime$, $\rest$, $\llcon$, $\cng{}$, $\cng{\abstime}$,  
$\cng{\mu}$ and $\bm{\updelta}$. This information confirms which
geometries are definitionally distinct from one another.}
\label{tab:concept-locations}
\begin{tabular}{c |  c | c | c | c | c }
               &  $\ca{\euclg}$  & $\ca{\relst}$ & 
$\ca{\galst}$ & $\ca{\newtst}$ & $\ca {\Lclst}$\\
\hline
$\cng{}$  &\cellcolor{green!42}  $\bm{\in}$ &\cellcolor{red!42} $\not\in$ &\cellcolor{red!42} $\not\in$ &\cellcolor{red!42} $\not\in$ & \cellcolor{green!42}
$\bm{\in}$\\
\hline
$\cng{\mu} $ &  \cellcolor{red!42} $\not\in$ &\cellcolor{green!42} $\bm{\in}$  & \cellcolor{red!42} $\not\in$ &\cellcolor{red!42} $\not\in$ &
\cellcolor{green!42} $\bm{\in}$\\
\hline
$\llcon $ & \cellcolor{red!42} $\not\in$ &\cellcolor{green!42} $\bm{\in}$ &\cellcolor{red!42} $\not\in$  & \cellcolor{red!42} $\not\in$ &
\cellcolor{green!42} $\bm{\in}$\\
\hline
$\cng{\abstime}$ & \cellcolor{red!42} $\not\in$  & \cellcolor{red!42} $\not\in$ &\cellcolor{green!42} $\bm{\in}$  &\cellcolor{green!42}
$\bm{\in}$ &\cellcolor{green!42} $\bm{\in}$\\
\hline
$\abstime$ & \cellcolor{red!42} $\not\in$  &  \cellcolor{red!42} $\not\in$ &\cellcolor{green!42} $\bm{\in}$ &\cellcolor{green!42} $\bm{\in}$ & \cellcolor{green!42}
$\bm{\in}$\\
\hline
$\rest$ & \cellcolor{red!42} $\not\in$ & \cellcolor{red!42} $\not\in$  & \cellcolor{red!42} $\not\in$ &\cellcolor{green!42} $\bm{\in}$ & \cellcolor{green!42}
$\bm{\in}$\\
\hline
$\bm{\updelta}$ & \cellcolor{red!42} $\not\in$ & \cellcolor{red!42} $\not\in$ &   \cellcolor{red!42} $\not\in$ & 
\cellcolor{red!42} $\not\in$ &\cellcolor{green!42} $\bm{\in}$\\
\hline
\end{tabular}
\end{table}

The next theorem tells us that, with just a few exceptions, if we
expand any of the geometries we have been studying with any of the
concepts from $\{\cng{},\cng{\mu},\llcon, \cng{\abstime}, \abstime,
\rest, \bm{\updelta}\}$ that are not concepts of the geometry in
question, the result is definitionally equivalent to $\Lclst$.

\begin{theorem}
\label{leiras2}

Let $\geometry\in\{\relst, \euclg, \galst, \newtst\}$ and
$\rel\in\{\cng{},\cng{\mu},\llcon, \cng{\abstime}, \abstime, \rest,
\bm{\updelta}\}$ and suppose that $\rel\not\in\ca{\geometry}$ and
$\rel\neq\rest$ if $\geometry=\galst$.
Table~\ref{tab:concept-locations} shows the possible choices of
$\geometry$ and $\rel$. Then:
\begin{enumerate}[(i)]
\item 
\label{LL1}
If $d\geq 3$, then $\la \geometry,\rel\ra$ is definitionally
equivalent to $\Lclst$, \ie\\
$\ca {\la \geometry,\rel\ra}=\ca{\Lclst}$.
\item
\label{LL2}
If $d=2$ and  $\la \geometry,\rel\ra\not\in\{\la
\euclg,\llcon\ra,  \la \euclg,\cng{\mu}\ra, 
\la\relst,\cng{}\ra\}$, then
 $\la \geometry,\rel\ra$ is definitionally
equivalent to $\Lclst$, \ie
$\ca {\la \geometry,\rel\ra}$ $=$ $\ca{\Lclst}$.
\item 
\label{LL3}
If $d=2$, then $\la\relst,\cng{}\ra$,
$\la\euclg, \cng{\mu}\ra$ and  
$\la\euclg, \llcon\ra$ are definitionally
equivalent, 
but they are not definitionally equivalent
to $\Lclst$.
\item 
\label{LL4}
$\la \OAff,\bm{\updelta}\ra$ is definitionally
equivalent to $\Lclst$, \ie
$\ca{\la \OAff,\bm{\updelta}\ra}$ $=$ \\ $\ca{\Lclst}$. \hfill $\Box$
\end{enumerate}
\end{theorem}

\subsection{Affine automorphisms and similarities} 

In this section, we show that the affine automorphisms of various
geometries can be characterised as similarity transformations. We
begin by defining the various types of similarities we will be using.

\begin{definition}[$\otimes$-similarities]
Suppose $A\colon\F^d\to\F^d$ is a  function and $\otimes$ is a
product on $F^d$. 
As a natural generalization of the notion to arbitrary 
ordered fields
  $\field$, we call $A$ 
a $\otimes$-\emph{similarity} 
if{}f it is an affine transformation
that respects squared $\otimes$-distances
up to scaling, \ie there is $a\in\F$ such that
\begin{equation}
\label{star-sim-eq}
\sqDist{A(\vec{p}\,)}{A(\vec{q}\,)}=a\cdot \sqDist{\vec{p}}{\vec{q}}
\quad\text{for every
  $\vec{p},\vec{q}\in\F^d$. }
 \end{equation}
The value $a$ is called the \emph{squared 
scale factor} (or \emph{square-factor}).
The set of all $\otimes$-similarities will be denoted by $\mathsf{Sim}_\otimes$.
\end{definition}

In particular, if $\otimes$ is the Euclidean scalar product, we call $A$ a
\emph{Euclidean similarity}, and if $\otimes$ is the Minkowski product, we call 
$A$ a \emph{Poincar\'e similarity}. 
When the square-factor 
$a$ is $1$, these similarities are traditionally called Euclidean and Poincar\'e \emph{transformations}.

\begin{remark}
\label{nonzero-rem}
Suppose $\otimes$ is a product. If there is a $\vec p$ for which $\vec
p \otimes \vec p \neq 0$, then the square-factor $a$ associated with
any $\otimes$-similarity $A$ must be non-zero.\footnote{ Because $A$
is a bijection, we can find some $\vec{q} \in F^d$ such that
$A(\vec{q}\,) = \vec{p}+A(\origin\,)$, and now \eqref{star-sim-eq}
yields $a \cdot \sqDist{\vec{q}}{\origin} =
\sqDist{A(\vec{q}\,)}{A(\origin\,)} =
\left(A(\vec{q}\,)-A(\origin\,)\right) \otimes
\left(A(\vec{q}\,)-A(\origin\,)\right) = \vec{p} \otimes \vec{p} \neq
0 $, \ie $a \neq 0$.}  In particular, the square-factor $a$ associated
with Euclidean and Poincar\'e similarities must be non-zero. In the
Euclidean case, $a$ is necessarily positive,\footnote{Because it can
be calculated from $a=\sqEd{A(\origin)}{A(\unit_1)}$, which is a
non-empty sum of squares.}  and if $d > 2$, $a$ must also be positive
in the Poincar\'e case.\footnote{ To see this, it is enough to show
that two non-zero `timelike' vectors cannot be Minkowski-orthogonal
(where a vector $\vec p$ is said to be `timelike' if $\vec p \MinkProd
\vec p > 0$). This latter follows easily from the Cauchy-Schwarz
inequality, and we omit the details.  }  However, this need not be the
case when $d=2$; for example, the transformation that interchanges
`time' and `space' in $F^2$ is a Poincar\'e similarity with $a = -1$.
\end{remark}

If $\field$ is a Euclidean field and $A$ a Euclidean similarity with
square-factor $a$, then $A$ is the composition of a Euclidean
transformation with the uniform scaling
$\vec{p}\mapsto\sqrt{a}\vec{p}$ with scale factor $\sqrt{a}$.  For
non-Euclidean fields this need not hold, \eg the map
$L(t,x)=(t+x,t-x)$ is a Euclidean similarity with square-factor 2
which cannot be decomposed in this way over the field of rationals.
We also note that, for any ordered field, the composition of a
Euclidean transformation with a uniform scaling of factor $\lambda
\neq 0$ is a Euclidean similarity with square-factor
$\lambda^2$. Analogous statements hold for Poincar\'e transformations
and similarities when $d>2$.

\begin{definition}[Galilean similarities]
A map $A\colon\F^d\rightarrow\F^d$ is called a 
\emph{Galilean similarity}
if{}f it is an affine transformation and
 there is $a\in \F$ such that
\begin{multline}
\label{gsim-eq}
\vec{p}\,\abstime\,\vec{q}\quad \Longrightarrow\\
A(\vec{p}\,)\,\abstime\, A(\vec{q}\,)\ \text{ and }\
\sqEd{A(\vec{p}\,)}{A(\vec{q}\,)}=a\cdot \sqEd{\vec{p}}
{\vec{q}}
\quad\text{for every
  $\vec{p},\vec{q}\in\F^d$. }
\end{multline}
We call $a$ in \eqref{gsim-eq} the
\emph{spatial square-factor}.
The affine transformation
$A$ is called a \emph{Galilean transformation}
if{}f it is 
a Galilean similarity with
spatial square-factor $1$ and in
addition, $A(\vec{p}\,)_1-A(\vec{q}\,)_1=p_1-q_1$ for all $\vec{p},\vec{q}\in\F^d$.
\end{definition}

Over the field of reals, the transformations that are called generalized Galilean transformations in \cite[p.51]{field80} are exactly the Galilean similarities that preserve orientation of time, \ie that have the property $A(\vec{p}\,)_1 > A(\vec{q}\,)_1$  if{}f  $p_1 > q_1$ for all $\vec{p},\vec{q}\in\F^d$.

The spatial square-factor associated with a Galilean similarity has to
be positive for the same reasons that the square-factor associated
with Euclidean similarities is positive.  If $\field$ is a Euclidean
field and $A$ is a Galilean similarity with spatial square-factor $a$,
then $A$ can be obtained as the composition of a Galilean
transformation, the spatial scaling $\vec{p}\mapsto\la
p_1,\sqrt{a}p_2, \ldots, \sqrt{a}p_d\ra$ and the temporal scaling
$\vec{p}\mapsto \la bp_1,p_2,\ldots,p_d\ra$ with (necessarily
non-zero) scale factor $b = A(\unit_1)_1-A(\origin)_1$.  
Conversely,
it is easy to see that any composition of a Galilean transformation
with non-zero spatial and temporal scalings is a Galilean similarity for arbitrary ordered field $\field$.

\begin{definition}[Trivial similarities]
A map is called a 
\emph{trivial Euclidean similarity} if{}f it
is a Euclidean similarity that respects $\rest$, and a 
\emph{trivial Galilean similarity} if{}f it is a 
Galilean similarity that respects $\rest$.\footnote{
We do not define trivial Poicar\'e similarities as well, because --
by Lemma~\ref{nevetkelladni-prop}\eqref{Ai3} --
the Poincar\'e similarities that
respect $\rest$ coincide with the trivial
Euclidean similarities we have already defined.
}
\end{definition}

Composition of a trivial Euclidean similarity with a non-zero temporal
scaling, or with a non-zero spatial scaling, is always a trivial
Galilean similarity.  It is also not difficult to prove, provided
$\field$ is Euclidean, that any trivial Galilean similarity $T$ can be
obtained as the composition of a trivial Euclidean similarity with a
non-zero temporal scaling, and as the composition of a (possibly
different) trivial Euclidean similarity with a non-zero spatial
scaling. Such a decomposition may not exist if $\field$ is not Euclidean.

The sets of Euclidean, Poincar\'e, Galilean, trivial Euclidean and
trivial Galilean similarities are denoted by $\EuclSim$, $\PoiSim$,
$\GenGal$, $\TrivSim$ and $\GenTriv$, respectively.

The following key theorem shows that these similarities are precisely
the affine automorphisms of their associated spacetimes. It is
precisely this result which makes it easy for us to check for
definitional equivalence between spacetimes. The proofs follow below
(Section \ref{sec:lemmas} onwards).

\begin{theorem}
\label{autgroups}
${}$
\begin{enumerate}[(i)]
\item $\AffAut{\euclg}=\EuclSim$.
\label{Beq1}
 \item
\label{Beq3}
 $\AffAut{\relst}=\AffAut{\minkst}=\PoiSim$. 
\item
\label{Beq4}
 $\AffAut{\galst}=\GenGal$.  
\item  
\label{Beq5}
$\AffAut{\newtst}=\GenTriv$.
\item 
\label{Beq6}
$\AffAut{\Lclst}=\TrivSim$.
\end{enumerate}
\end{theorem}

Because $\aut{\Model}$ is always a group under composition for any
model $\Model$, it immediately follows that all of these similarity
sets are groups.

\begin{corollary}
\label{group-cor}
$\EuclSim$, $\PoiSim$, $\GenGal$, $\GenTriv$
and $\TrivSim$ all form groups under composition.
\end{corollary}

\subsection{Lemmas}
\label{sec:lemmas}

The lemma below shows that a product and its associated squared distance function are interdefinable (we already know that the squared distance is defined in terms of the product).
\begin{lemma} 
\label{interdefinable}
Suppose $\vec{p}, \vec{q} \in F^d$ and let $\otimes$ be a product on $F^d$. Then $\otimes$ is definable from $\mathsf{d}_\otimes^2$ by
\[
(\vec{p} \otimes \vec{q}\,) = 
\frac{\sqDist{\vec{p}}{\origin} + 
\sqDist{\vec{q}}{\origin}
-\sqDist{\vec{p}}{\vec{q}} 
}{2}.
\]
\end{lemma}
\begin{proof}
Because $\otimes$ is bilinear and symmetric, we have
\[
 \sqDist{\vec{p}}{\vec{q}} 
	  = (\vec{p} - \vec{q}\,) \otimes (\vec{p} - \vec{q}\,) \\
		% = \vec{p}\otimes\vec{p} - 2(\vec{p}\otimes\vec{q}\,) + \vec{q}\otimes\vec{q}
		= \sqDist{\vec{p}}{\origin} - 2(\vec{p}\otimes\vec{q}\,) + \sqDist{\vec{q}}{\origin},
\]
and the claim now follows by rearrangement because $\field$, being an
ordered field, has characteristic 0.
\end{proof}

\begin{lemma}
\label{same-a-lemma}
Suppose $\otimes$ is a product on $F^d$ and $L$ is a linear bijection. Then the following conditions are equivalent for any $a \in F$:
\begin{enumerate}[(i)]
\item 
	$\sqDist{L(\vec{p}\,)}{L(\vec{q}\,)} = a \cdot \sqDist{\vec{p}}{\vec{q}}$ for all $\vec{p}, \vec{q} \in F^d$,
\item 
	$L(\vec{p}\,)\otimes L(\vec{q}\,)= a \cdot (\vec{p} \otimes \vec{q}\,)$ for all $\vec{p}, \vec{q} \in F^d$,
\item 
	$L(\unit_i)\otimes L(\unit_j)= a \cdot (\unit_i \otimes \unit_j)$ for all $i, j \in \{1,\dots,d\}$.
\end{enumerate}
\end{lemma}
\begin{proof}
\underline{(i)$\Rightarrow$(ii)}: 
Suppose $a$ satisfies (i).
	Using Lemma \ref{interdefinable} to 
substitute for $\otimes$, 
and noting that $L(\origin\,) = \origin$ because 
$L$ is linear, gives
	\[ 
\begin{aligned}
	 L(\vec{p}\,) \otimes L(\vec{q}\,) 
&= \frac{\sqDist{L(\vec{p}\,)}{\origin} +
 \sqDist{L(\vec{q}\,)}{\origin} -
\sqDist{L(\vec{p}\,)}{L(\vec{q}\,)}}{2}\\
&= \frac{\sqDist{L(\vec{p}\,)}{L(\origin\,)} +
 \sqDist{L(\vec{q}\,)}{L(\origin\,)} -
\sqDist{L(\vec{p}\,)}{L(\vec{q}\,)}}{2}\\  
	  &= 
\frac{
a \cdot \sqDist{\vec{p}}{\origin} + 
a \cdot \sqDist{\vec{q}}{\origin} -a \cdot \sqDist{\vec{p}}{\vec{q}} }{2} \\
	  &= a \cdot (\vec{p} \otimes \vec{q}\,).
		\end{aligned} \]
\\
\underline{(ii)$\Rightarrow$(i)}: 
Suppose $a$ satisfies (ii).
Because $L$ is linear, we have
\begin{multline*}
\sqDist{L(\vec{p}\,)}{L(\vec{q}\,)} = 
\left(L(\vec{p}\,)-L(\vec{q}\,)\right)
 \otimes \left(L(\vec{p}\,)-L(\vec{q}\,)\right)=\\
L(\vec{p}-\vec{q}\,) \otimes L(\vec{p}-\vec{q}\,) = a \left[ (\vec{p}-\vec{q}\,) \otimes (\vec{p}-\vec{q}\,) \right]
 = a \cdot \sqDist{\vec{p}}{\vec{q}}.
\end{multline*}
\\
\underline{(ii)$\Rightarrow$(iii)} 
Suppose $a$ satisfies (ii). Then (iii) follows immediately on taking $\vec{p} = \unit_i$ and $\vec{q} = \unit_j$.
\\
\underline{(iii)$\Rightarrow$(ii)}: 
Suppose $a$ satisfies (iii). 
Writing $\vec{p} = \sum_i{p_i\unit_i}$ and 
$\vec{q} = \sum_j{q_j\unit_j}$, the result follows 
straightforwardly from the bilinearity of $\otimes$ 
and linearity of $L$.
\end{proof}

Re-expressing this in terms of $\otimes$-similarities gives us the following.
\begin{corollary} 
\label{star-sim-lemma}
Let $L$ be a linear bijection and suppose $a \in \F$. Then the following statements are equivalent:
\begin{enumerate}[(i)]
\item $L \in \mathsf{Sim}_\otimes$ and the square-factor of $L$ is $a$.
\item
For all $\vec{p}, \vec{q} \in F^d$,
$
	L(\vec{p}\,)\otimes L(\vec{q}\,)= a \cdot (\vec{p} \otimes \vec{q}\,)
$.
\item 
For all $i, j \in \{1,\dots,d\}$,
$
	L(\unit_i)\otimes L(\unit_j)= a \cdot (\unit_i \otimes \unit_j)
$.\qed
\end{enumerate}
\end{corollary}

Upon calculating $\unit_i\otimes\unit_j$ for the corresponding products, Corollary~\ref{star-sim-lemma} gives us the following two lemmas immediately:

\begin{lemma}
\label{EuclSim-lemma}
Let $L$ be a linear bijection and suppose $a \in \F$. Then the following  statements are equivalent:
\begin{enumerate}[(i)] 
\item
	$L\in\EuclSim$ and the square-factor of $L$ is $a$.
\item
For all $\vec{p}, \vec{q} \in F^d$,
$
	L(\vec{p}\,)\ScalProd L(\vec{q}\,)= a \cdot (\vec{p} \ScalProd \vec{q}\,)
$.
\item 
\label{EuclSim-units}
For all $i, j \in \{1,\dots,d\}$,
	\begin{equation}
\label{ESu}
	L(\unit_i)\ScalProd L(\unit_j)=\begin{cases}
	a \text{ if $i=j$},\\ 
	0 \text{ if $i\neq j$}.
	\end{cases}
	\end{equation} \qed
\end{enumerate}
\end{lemma}

\begin{lemma}
\label{PoiSim-lemma} 
Let $L$ be a linear bijection and suppose $a \in \F$. Then the following  statements are equivalent:
\begin{enumerate}[(i)] 
\item $L\in\PoiSim$ and the square-factor of $L$ is $a$.
\item
For all $\vec{p}, \vec{q} \in F^d$,
$
	L(\vec{p}\,)\MinkProd L(\vec{q}\,)= a \cdot (\vec{p} \MinkProd \vec{q}\,)
$.
\item  
\label{PS-units}
For all $i, j \in \{1,\dots,d\}$,
\begin{equation}
\label{PSu}
L(\unit_i)\MinkProd L(\unit_j)=\begin{cases}
a \text{ if $i=j=1$},\\ 
-a \text{ if $i=j\neq 1$},\\
0 \text{ if $i\neq j$}.
\end{cases}
\end{equation}\qed
\end{enumerate}
\end{lemma}

\begin{definition}
The \emph{time axis} $\taxis\subseteq\F^d$ and 
the \emph{simultaneity at 
the origin} $\simult\subseteq\F^d$ 
are   defined by
\begin{align*}
\taxis & \defeq\{\laa t,0,\ldots,0\raa : t\in\F\},\\
\simult & \defeq\{\vec{p}\in\F^d : p_1=0\}.
\end{align*}
\end{definition}
We note that $\vec{p}\,\abstime\,\vec{q}\; \Leftrightarrow\;
\vec{p}-\vec{q}\in\simult$ and $\vec{p}\,\rest\,\vec{q}\; \Leftrightarrow\;
\vec{p}-\vec{q}\in\taxis$.

\begin{lemma}
\label{LemProd}
Assume that $\otimes$ is the Euclidean scalar product $\ScalProd$ or
Minkowski product $\MinkProd$. Then
for any $\vec{p},\vec{q}\in\F^d$ the following items
hold:
\begin{enumerate}[(i)]
\item
\label{harmas}
($\vec{p}\in\taxis\setminus\{\origin\}\;\text{ and }\;
\vec{p}\otimes\vec{q}=0)\ \ \Longrightarrow\ \ \vec{q}\in\simult$.
\item
\label{MerEq}
$(\forall\vec{q}\in\simult)\ \vec{p}\otimes 
\vec{{q}}=0
\ \ \Longrightarrow\ \ \vec{p}\in\taxis$.
\item
\label{hetes}
$(\vec{p}\in\taxis\text{ and }\vec{q}\in\simult)\ 
\ \Longrightarrow\ \
\vec{p}\otimes\vec{q}=0$.

\item 
\label{same}
$\vec{p},\vec{q}\in\taxis\ 
\Longrightarrow\ \vec{p}\ScalProd\vec{q}=
\vec{p}\MinkProd\vec{q}$.
\item  
\label{MinkScalE}
$\vec{p},\vec{q}\in\simult\ \ \Longrightarrow\ \
\vec{p}\ScalProd\vec{q}=-
\left(\vec{p}\MinkProd\vec{q}\,\right)$.
\item \label{EuclMinkProdEq}
$\vec{p}\ScalProd\vec{q}=\vec{p}\MinkProd\vec{q}=0\ 
\ \Longrightarrow\ \ 
(\vec{p}\in\simult\text{ or }\vec{q}\in\simult)$.
\end{enumerate}
\end{lemma}

\begin{proof}
\underline{\eqref{harmas}, \eqref{hetes}, \eqref{same}, \eqref{MinkScalE}}:
These all follow directly from the definitions
of the Euclidean scalar product and Minkowski product.
Item \underline{\eqref{MerEq}} also follows from these definitions 
as for every $i\in \{2,\ldots,d\}$, we have that
 $\unit_i\in\simult$ and $\vec{p}\otimes\unit_i=0\; 
\Rightarrow\; p_i=0$. 

\proofof{\eqref{EuclMinkProdEq}}:
Assume that 
$
\vec{p}\ScalProd\vec{q}=\vec{p}\MinkProd\vec{q}=0.
$
Then
\begin{align*}
p_1q_1+p_2q_2+\ldots+p_dq_d & =0 \text{ and}\\
p_1q_1-p_2q_2-\ldots-p_dq_d & =0.
\end{align*}
By adding the equations above, we get that
$2p_1q_1=0$. 
As the characteristic of every ordered field is $0$, we get that $p_1q_1=0$, and this implies that
$p_1=0$ or $q_1=0$ because $\field$ is a field.
Thus \eqref{EuclMinkProdEq} 
holds.
\end{proof}

\begin{lemma}
\label{kozos}
Suppose that $L$ is a linear bijection and
$L\in\EuclSim\cup\PoiSim$.
Then the following items are equivalent.
\begin{enumerate}[(i)]
\item
\label{KK1}
$L[\simult]\subseteq\simult$.
\item
\label{KK2}
$L[\simult]=\simult$.
\item
\label{KK3}
$L[\taxis]\subseteq \taxis$
\item
\label{KK4}
$L[\taxis]=\taxis$.
\end{enumerate}
\end{lemma}
\begin{proof}
As $L$ is a linear bijection and $\taxis$ and $\simult$ are 
subspaces of $\F^d$, we have immediately that
\eqref{KK1}$\Leftrightarrow$\eqref{KK2} and
\eqref{KK3}$\Leftrightarrow$\eqref{KK4}.
To complete the proof, it is therefore enough to show that \eqref{KK3}$\Rightarrow$\eqref{KK1}
and \eqref{KK2}$\Rightarrow$\eqref{KK3}.
\smallskip

Let $\otimes$ be the Euclidean scalar product if
$L\in\EuclSim$ and the Minkowski product 
otherwise.
Then by Lemmas~\ref{EuclSim-lemma} and \ref{PoiSim-lemma},
we have that
\begin{equation}
\label{orth}
\vec{p}\otimes\vec{q}=0\ \Longrightarrow\ L(\vec{p}\,)
\otimes L(\vec{q}\,)=0.
\end{equation}

\noindent
\proofof{\eqref{KK3} $\Rightarrow$ \eqref{KK1}}:
Assume that $L[\taxis]\subseteq\taxis$.
We want to prove that $L[\simult]\subseteq\simult$. 
To prove this
let $\vec{p}\in\simult$. Then, by Lemma~\ref{LemProd}\eqref{hetes},
$\unit_1\otimes\vec{p}=0$, and so $L(\unit_1)\otimes 
L(\vec{p}\,)=0$ by \eqref{orth}.
By our assumption that $L[\taxis]\subseteq\taxis$
and the fact that $L$ is a linear bijection,
we have that 
$L(\unit_1)\in\taxis\setminus\{\origin\}$.
Therefore, by Lemma~\ref{LemProd}\eqref{harmas},
$L(\vec{p}\,)\in\simult$ as required. 

\smallskip
\proofof{\eqref{KK2}$\Rightarrow$\eqref{KK3}}:
Assume that $L[\simult]=\simult$.
We want to prove that $L[\taxis]\subseteq\taxis$.
To prove this, let $\vec{p}\in\taxis$.
We will show that
\begin{equation}
\label{vagy}
(\forall \vec{q}\in\simult)\  L(\vec{p}\,)\otimes\vec{q}=0,
\end{equation}
whence $L(\vec{p}\,)\in\taxis$ will 
follow by Lemma~\ref{LemProd}\eqref{MerEq}.

Let $\vec{q}\in\simult$. Then 
$L^{-1}(\vec{q}\,)\in\simult$ because $L[\simult]=\simult$.
So $\vec{p}\otimes L^{-1}(\vec{q}\,)=0$, and hence
$L(\vec{p}\,)\otimes\vec{q} = 0$ by \eqref{orth}. 
Thus \eqref{vagy} holds, as claimed. 

This completes the proof.
\end{proof}

\begin{lemma}${}$
\label{nevetkelladni-prop}
\begin{enumerate}[(i)] 
\item 
\label{Ai4}
$\EuclSim\cap\aut{\la\F^d,\abstime\ra}=
\EuclSim\cap\aut{\la \F^d,\rest\ra}
=\TrivSim$. 
\item
\label{Ai3}
$\PoiSim\cap\aut{\la  \F^d,\abstime\ra}=
\PoiSim\cap\aut{\la \F^d,\rest\ra}= \TrivSim$.
\item
\label{Ai5}
$\EuclSim\cap\PoiSim=\TrivSim$ if $d>2$.
\item
\label{Ai2}
$\GenGal\subseteq\aut{\la \F^d, \abstime\ra}$.

\item
\label{Ai2.5}
$\TrivSim\subseteq\GenTriv\subseteq
\aut{\la\F^d,\abstime,\rest\ra}$.

\item
\label{Ai6}
$\PoiSim\cap\GenGal=\TrivSim$.

\item
\label{Ai7}
$\EuclSim\cap\GenGal=\TrivSim$.
\end{enumerate} 
\end{lemma}

\begin{proof}
It is easy to see that for every linear bijection $L$, we have
\begin{equation}
\label{refor}
L\in\aut{\la \F^d,\abstime\ra}\ \Leftrightarrow\ L[\simult]=\simult\
\text{ and }\ L\in\aut{\la\F^d,\rest\ra}\ \Leftrightarrow\ L[\taxis]=\taxis.
\end{equation}
From this, together with Lemma~\ref{kozos} and the fact that all translations 
are contained in each of sets
$\EuclSim$, $\PoiSim$, $\aut{\la \F^d,\abstime\ra}$ and
$\aut{\la\F^d,\rest\ra}$, we have
\begin{align}
\label{EuSR}
\EuclSim\cap\aut{\la\F^d,\abstime\ra} & =
\EuclSim\cap\aut{\la \F^d,\rest\ra} 	\\
\label{PoSR}
\PoiSim\cap\aut{\la  \F^d,\abstime\ra} & =
\PoiSim\cap\aut{\la \F^d,\rest\ra}.
\end{align}

\smallskip
\proofof{\eqref{Ai4}}:
By definition, $\TrivSim=\EuclSim\cap\aut{\la \F^d,\rest\ra}$; so \eqref{Ai4} follows 
immediately from \eqref{EuSR}.

\smallskip
\proofof{\eqref{Ai3}}:
We have already seen, \eqref{PoSR}, that $\PoiSim\cap\aut{\la  \F^d,\abstime\ra}
= \PoiSim\cap\aut{\la \F^d,\rest\ra}$; we have to show that these sets
are equal to $\TrivSim$.

First, we prove that  
$\PoiSim\cap\aut{\la\F^d,\rest\ra} \subseteq \TrivSim$.
Let $L\in\PoiSim\cap\aut{\la\F^d,\rest\ra}$.
To prove that $L\in\TrivSim$, it is sufficient to prove that
$L\in\EuclSim$ as $\TrivSim=\EuclSim\cap\aut{\la \F^d,\rest\ra}$.
We can, as usual, assume that $L$ is a linear
bijection because $\PoiSim$ and $\EuclSim$ contain
all the translations, they are closed under
composition and translations respect $\rest$.
By \eqref{refor} and \eqref{PoSR}, we have
$L[\taxis]=\taxis$ and $L[\simult]=\simult$, and so
\begin{equation}
\label{long}
L(\unit_1)\in\taxis 
\quad\text{ and }\quad
L(\unit_i)\in\simult \text{ if } i\in\{2,\dots,d\} .
\end{equation}

Since $L \in \PoiSim$, Lemma~\ref{PoiSim-lemma} tells us there is an $a\in\F$
satisfying \eqref{PSu}, \ie
$L(\unit_i)\MinkProd L(\unit_j) = a$ if $i=j=1$; $-a$ if $i=j\neq 1$; and
$0$ if $i\neq j$. We have to prove the corresponding $\EuclSim$-based equation, 
\eqref{ESu}, holds for $L$,
which we can easily do using the basic facts listed in Lemma~\ref{LemProd}:

\underline{\textit{case 1}: $i=j=1$}:
According to \eqref{long}, we have $L(\unit_1) \in \taxis$. Therefore,
$L(\unit_1) \ScalProd L(\unit_1) = L(\unit_1) \MinkProd L(\unit_1) = a$ by
\eqref{PSu} and \ref{LemProd}\eqref{same}.

\underline{\textit{case 2}: $i=j\neq 1$}:
According to \eqref{long}, we have $L(\unit_i) \in \simult$. Therefore,
$L(\unit_i) \ScalProd L(\unit_i) =  -L(\unit_i) \MinkProd L(\unit_i) = -(-a) = a$ by
\eqref{PSu} and \ref{LemProd}\eqref{MinkScalE}.

\underline{\textit{case 3}: $i \neq j$}:
If $i=1$ (or analogously, if $j=1$), we have $\unit_i \in \taxis$ and $\unit_j \in \simult$, whence $L(\unit_i) \in \taxis$
and $L(\unit_j) \in \simult$ by \eqref{long}, and $L(\unit_i) \ScalProd L(\unit_j) = 0$
by \ref{LemProd}\eqref{hetes}.
Finally, suppose
$i, j \neq 1$. Then both $L(\unit_i)$ and $L(\unit_j)$ are in $\simult$, and so
$L(\unit_i) \ScalProd L(\unit_j) = - L(\unit_i) \MinkProd L(\unit_j) = 0$
by \eqref{PSu} and \ref{LemProd}\eqref{MinkScalE}.

As required, we have successfully shown that 
$L(\unit_i) \ScalProd L(\unit_j) = a$
when $i=j$ and $0$ otherwise, whence it follows by
Lemma~\ref{EuclSim-lemma} that $L\in\EuclSim$.
This shows that
\begin{equation}
\label{lila}
\PoiSim\cap\aut{\la\F^d,\rest\ra}\subseteq\TrivSim.
\end{equation}

It remains to prove the reverse inclusion,
\begin{equation}
\label{elil}
\PoiSim\cap\aut{\la\F^d,\rest\ra}\supseteq\TrivSim ,
\end{equation} 
so choose any $L\in\TrivSim$, and assume wlog that $L$ is a linear
bijection. By \eqref{refor} and 
item \eqref{Ai4} of the present lemma, 
we have that 
$L[\taxis]=\taxis$ and $L[\simult]=\simult$,
and so \eqref{long} holds, \ie
$L(\unit_1)\in\taxis$ and $L(\unit_i)\in\simult$ 
when $i \neq 1$.

As $\TrivSim\subseteq\EuclSim$, 
Lemma~\ref{EuclSim-lemma} tells us there is some $a$ satisfying \eqref{ESu},
\ie $L(\unit_i)\ScalProd L(\unit_j) = a$ if $i=j$, and $0$ otherwise; and by a
 similar three-case argument to that just given, it is easy to see that the 
corresponding $\PoiSim$-based equation,
\eqref{PSu}, holds for $L$.
It therefore follows, by
Lemma~\ref{PoiSim-lemma}, that $L\in\PoiSim$.
As $L$ also respects $\rest$ (by definition of $\TrivSim$), 
we have that $L\in\PoiSim\cap\aut{\la\F^d,\rest\ra}$, and this shows
 that the reverse inclusion, \eqref{elil}, holds as required.

\proofof{\eqref{Ai5}}:
Assume $d>2$. By
item \eqref{Ai3} of the present lemma, 
we have that
$\TrivSim\subseteq\EuclSim\cap\PoiSim$.
Thus to prove \eqref{Ai5} it remains to prove that
$\EuclSim\cap\PoiSim\subseteq\TrivSim$, so choose
any $L\in\EuclSim\cap\PoiSim$.

As before, we can assume wlog that $L$ is 
a linear bijection. We want
to prove that $L\in\TrivSim$, and since we have already shown (item
\eqref{Ai4} of the present lemma),
that $\TrivSim=\EuclSim\cap\aut{\la \F^d,\abstime\ra}$,
it is enough to show that $L$ respects $\abstime$.

Note first that we can choose some $k$ ($1\leq k\leq d$) 
such that
$L(\unit_k)\not\in\simult$; otherwise,
we would have $L(\unit_k)\in\simult$ for all $k$, 
and by linearity of $L$, $L[\F^d]\subseteq\simult$
would then hold, which contradicts the fact that $L$ is 
bijective.

As $L$ is both a Euclidean and a Poincar\'e
similarity, Lemmas~\ref{EuclSim-lemma} 
and \ref{PoiSim-lemma} tell us that, for this $k$ and every $i\neq k$,
$
L(\unit_k)\ScalProd L(\unit_i) = 
L(\unit_k)\MinkProd L(\unit_i)=0.
$
As $L(\unit_k)\not\in\simult$, it follows by
Lemma~\ref{LemProd}\eqref{EuclMinkProdEq} that 
\begin{equation}
\label{bennevanE}
L(\unit_i)\in\simult\text{ for every $i\neq k$}.
\end{equation}

 We will show that $k = 1$ and hence, 
by linearity, bijectivity and \eqref{bennevanE}, that
$L[\simult] = \simult$.
To this end, choose any 
$j \in \{2,\dots,d\}\setminus\{k\}$.
Such $j$ exists because $d\geq 3$; and because $j \neq k$, we have from \eqref{bennevanE} 
that 
$
L(\unit_j)\in\simult.
$
Using the
fact that $L$ is a linear bijection, it now follows that
\begin{align*}
L(\unit_1)\ScalProd L(\unit_1) &=
       L(\unit_j)\ScalProd L(\unit_j) && 
\text{(by $L\in\EuclSim$ and Lemma~\ref{EuclSim-lemma})}\\
     &= -\left(L(\unit_j)\MinkProd L(\unit_j)\right) &&
\text{(by $L(\unit_j)\in\simult$ and 
 Lemma~\ref{LemProd}\eqref{MinkScalE})}\\
  &= L(\unit_1)\MinkProd L(\unit_1) &&
\text{(by $L\in\PoiSim$, Lemma \ref{PoiSim-lemma} and $j\neq 1$).}  
\end{align*}

This implies that $L(\unit_1) \not\in \simult$, 
because otherwise
Lemma~\ref{LemProd}\eqref{MinkScalE} 
would entail $L(\unit_1) \ScalProd  L(\unit_1)
= - ( L(\unit_1) \ScalProd L(\unit_1))$, and 
hence (because $\field$ is ordered)
that $L(\unit_1) \ScalProd L(\unit_1) = 0$.
However, $L(\unit_1) \ScalProd L(\unit_1)>0$ 
because it is a non-empty sum of squares as
by bijectivity and linearity of $L$,  
$L(\unit_1)\neq\origin$.

Looking at \eqref{bennevanE}, it now follows as claimed that $k=1$ (and thus
$L[\simult]=\simult$), and hence by \eqref{refor} that 
$L\in\aut{\la\F^d,\abstime\ra}$. 
That is, $L$ respects $\abstime$ as required.
This completes the proof of \eqref{Ai5}.

\proofof{\eqref{Ai2}}:
It is easy to see by \eqref{gsim-eq}
that composition
of a Galilean similarity and 
a translation is a Galilean similarity because
translations respect $\abstime$ and
do not change squared Euclidean distances.
Let $A\in\GalSim$. Since $A$ is affine, we 
can write
$A=\tau\circ L$, where $L$ is a linear 
bijection and $\tau$ a translation.
Now $L=\tau^{-1}\circ A$, and since $\tau^{-1}$ is also
a translation, we have $L\in\GalSim$. To
prove that $A\in\aut{\la\F^d,\abstime\ra}$, 
it is enough to prove that 
$L\in\aut{\la\F^d,\abstime\ra}$ 
because translations respect $\abstime$.
We have that $L[\simult]\subseteq\simult$
because $L(\origin\,)=\origin$,
every point in $\simult$ is
$\abstime$-related to $\origin$,
$L$ takes
$\abstime$-related points to $\abstime$-related 
points (by $L\in\GalSim$) and any point
$\abstime$-related to $\origin$ is
in $\simult$. Then $L[\simult]=\simult$
because $L$ is a bijective linear
transformation.
By 
\eqref{refor}, 
$L\in\aut{\la\F^d,\abstime\ra}$ follows as required.

\proofof{\eqref{Ai2.5}}:
Recall that 
$\GenTriv \defeq \GenGal\cap\aut{\la \F^d,\rest\ra}$.
From this and item \eqref{Ai2}, just proven, which says that
$\GenGal\subseteq\aut{\la\F^d,\abstime\ra}$, 
we have immediately that $\GenTriv\subseteq\aut{\la\F^d,\abstime,\rest\ra}$.

To prove that $\TrivSim\subseteq\GenTriv$,
let $A\in\TrivSim$ and recall that $\TrivSim \defeq \EuclSim\cap\aut{\la\F^d,\rest\ra}$.
As $A\in\EuclSim$, there is $a\in\F$ such that
$\sqEd{A(\vec{p}\,)}{A(\vec{q}\,)}=
a\cdot \sqEd{\vec{p}}{\vec{q}}$ for every 
$\vec{p},\vec{q}\in\F^d$.
Moreover, we know that $A$ respects $\abstime$ by item \eqref{Ai4}.
Therefore, \eqref{gsim-eq} is satisfied, \ie 
$A\in\GenGal$. Thus $A$ is an element of $\GenGal$ which respects $\rest$, \ie $A\in\GenTriv$,
and so $\TrivSim\subseteq\GenTriv$, as claimed.
This completes the proof of \eqref{Ai2.5}.

\proofof{\eqref{Ai6}}:
By \eqref{Ai2} and \eqref{Ai3}, we have that
\[
\PoiSim\cap\GenGal\subseteq (\PoiSim
\cap\aut{\la \F^d,\abstime\ra})=\TrivSim.
\]
Conversely, by \eqref{Ai3} and \eqref{Ai2.5}, we have that
$\TrivSim\subseteq\PoiSim$ and
$\TrivSim\subseteq\GenGal$.  
Therefore,
\[ 
\TrivSim\subseteq\PoiSim\cap\GenGal.
\]
Thus $\PoiSim\cap\GenGal=\TrivSim$ as claimed. 
This completes the proof of \eqref{Ai6}.

\proofof{\eqref{Ai7}}:
By \eqref{Ai2.5}, $\TrivSim\subseteq\GalSim$, 
and by construction $\TrivSim\subseteq\EuclSim$. 
Therefore,
\[
\TrivSim\subseteq\EuclSim\cap\GalSim .
\]
Conversely, by \eqref{Ai2} and \eqref{Ai4}, we have
\[
\EuclSim\cap\GalSim\subseteq\EuclSim\cap
\aut{\la \F^d,\abstime\ra}=\TrivSim .
\]
Thus $\EuclSim\cap\GalSim=\TrivSim$, as claimed.

This completes the proof of \eqref{Ai7} 
and hence of the Lemma as a whole.
\end{proof}

\subsection{Proof of Theorem~\ref{autgroups}  
(page \pageref{autgroups})}

\begin{proofofpart}{\ref{autgroups}\eqref{Beq1} 
$(\AffAut{\euclg}=\EuclSim)$}
It is easy to see from the definitions that
Euclidean similarities take
congruent segments to congruent segments, 
and (because their square-factors are nonzero, 
cf.\ Remark~\ref{nonzero-rem}) non-congruent
segments to non-congruent ones. As one would expect, therefore, Euclidean
similarities respect the Euclidean congruence relation, $\cng{}$. They
also respect $\Bw$, because they are affine transformations. 
It follows from this that 
$\EuclSim\subseteq\AffAut{\euclg}$.

To prove conversely that $\AffAut{\euclg}\subseteq\EuclSim$,
choose any $A\in\AffAut{\euclg}$ and write $A=\tau\circ L$
where $L$ 
is a linear bijection and $\tau$
a translation. 
We want to prove that $A\in\EuclSim$, but since we already 
know that $\tau\in\EuclSim$, it is
enough to show that $L\in\EuclSim$.
We will do this by showing that
$L$ satisfies item \eqref{EuclSim-units} 
of Lemma~\ref{EuclSim-lemma}. 

Observe that $L\in\AffAut{\euclg}$ because both $A$ and $\tau$ are 
in $\AffAut{\euclg}$.

We begin by proving that for every distinct $i,j$, we have
that $L(\unit_i)\ScalProd L(\unit_j)=0$. To see this, let $i,j$ be
distinct, and observe that
$(\unit_i,\unit_j)\cng{}(\unit_i,-\unit_j)$. This implies that
$(L(\unit_i),L(\unit_j))\cng{}(L(\unit_i), -L(\unit_j))$ because
$L$ is a linear transformation and an automorphism of $\euclg$. 
By this, we have that
\[
(L(\unit_i)-L(\unit_j))\ScalProd (L(\unit_i)-L(\unit_j)) =
(L(\unit_i)+L(\unit_j))\ScalProd (L(\unit_i)+L(\unit_j)).
\]
Using the bilinearity and symmetry of the scalar product
$\ScalProd$, this implies that
\begin{multline*}
L(\unit_i)\ScalProd L(\unit_i)-
2 L(\unit_i)\ScalProd L(\unit_j)+
L(\unit_j)\ScalProd L(\unit_j)=\\
L(\unit_i)\ScalProd L(\unit_i)+2 L(\unit_i)\ScalProd L(\unit_j)+
L(\unit_j)\ScalProd L(\unit_j)
\end{multline*}
and hence $4 \left(L(\unit_i)\ScalProd L(\unit_j)\right)=0$. Because
the characteristic of an ordered field is $0$, it follows that
$L(\unit_i)\ScalProd L(\unit_j)=0$.

It remains to find an $a\in\F$ such that 
$L(\unit_i)\ScalProd L(\unit_i)=a$
for all $i$.
Choose any $k$ ($1\leq k\leq d$) and 
set
$a:=L(\unit_k)\ScalProd L(\unit_k)$.
Now choose any $i$. We will prove that  
$L(\unit_i)\ScalProd L(\unit_i)=a$.
Clearly $(\unit_i,\origin\, )\cng{}
(\unit_{k},\origin\,)$.
This implies that 
$(L(\unit_i),\origin\,)\cng{}
(L(\unit_{k}),
\origin\,)$
because $L$ is an automorphism of 
$\euclg$ and $L(\origin\,)=\origin$.
So $L(\unit_i)\ScalProd L(\unit_i)= L(\unit_k)
\ScalProd
L(\unit_{{k}})=a$ as claimed, 
and $L\in\EuclSim$ follows. 
\end{proofofpart}

\begin{proofofpart}{\ref{autgroups}\eqref{Beq3} $(\AffAut{\relst}=\AffAut{\minkst}=\PoiSim)$}
It is easy to see from the definitions that
$\PoiSim\subseteq\AffAut{\minkst}$, using an argument analogous
to that just given in item \eqref{Beq1} to show that $\EuclSim\subseteq\AffAut{\euclg}$.
We also have that $\AffAut{\minkst}\subseteq\AffAut{\relst}$ because
$\llcon$ is definable in terms of $\cng{\mu}$ (as
 $\vec{p}\,\llcon\,\vec{q}\ \Leftrightarrow\ 
(\vec{p},\vec{q}\,)\cng{\mu}
(\vec{p},\vec{p}\,)$). This shows that
\[
    \PoiSim \subseteq \AffAut{\minkst} \subseteq \AffAut{\relst}.
\]
So to complete the proof of \eqref{Beq3}, we
need only show that $\AffAut{\relst}\subseteq\PoiSim$. 
For the case $d>2$, this follows easily from Lester~\cite{Lester}.
Here, we give a self-contained proof
 which works for $d=2$, too.

Given any $\vec{p}$ and
$\vec{q}$, by \eqref{llcon-alt}, we have
\begin{equation}
\label{e-llcon}
\vec{p}\,\llcon\,\vec{q}\ \Longleftrightarrow\ (\vec{p}-\vec{q}\,)\MinkProd
(\vec{p}-\vec{q}\,)=0.
\end{equation}

To prove $\AffAut{\relst}\subseteq\PoiSim$,
choose any $A\in\AffAut{\relst}$
and assume $A=\tau\circ L$
where $L$ is a linear bijection and $\tau$
is a translation.
Because $A$ and $\tau$ are
automorphisms of $\relst$, so is $L$; and to
prove that $A\in\PoiSim$, it is enough to 
show that $L\in\PoiSim$ because $\tau\in\PoiSim$.
We will show this by showing that $L$ 
satisfies item \eqref{PS-units}
of Lemma~\ref{PoiSim-lemma}. 

First, we prove, for all $i \in  \{2,\dots,d\}$, that
\begin{equation}
\label{e1-bizonyitando}
L(\unit_1)\MinkProd L(\unit_i)=0 \quad\text{ and }\quad
L(\unit_1)\MinkProd L(\unit_1)=- \left(L(\unit_i)\MinkProd L(\unit_i)\right).
\end{equation} 
Choose any $i \in \{2,\dots,d\}$ and observe that
$\unit_1\,\llcon\,\unit_i$ and $\unit_1\,\llcon\, (-\unit_i)$.  From
this it follows that $L(\unit_1)\,\llcon\, L(\unit_i)$ and
$L(\unit_1)\,\llcon\, \left(-L(\unit_i)\right)$ because $L$ is a
linear transformation and an automorphism of $\relst$. Thus, by
\eqref{e-llcon}, we have that
$\left(L(\unit_1)-L(\unit_i)\right)\MinkProd\left(L(\unit_1)-
L(\unit_i)\right)=0$ and $\left(L(\unit_1)+L(\unit_i)\right)
\MinkProd\left(L(\unit_1)+ L(\unit_i)\right)=0$. From these, using the
fact that $\MinkProd$ is a symmetric bilinear function, we get that
\begin{align*}
  & L(\unit_1)\MinkProd L(\unit_1)-2\left(L(\unit_1)\MinkProd
  L(\unit_i)\right) +L(\unit_i)\MinkProd L(\unit_i) = 0,\\ &
  L(\unit_1)\MinkProd L(\unit_1)+2\left(L(\unit_1)\MinkProd
  L(\unit_i)\right) +L(\unit_i)\MinkProd L(\unit_i) = 0.
\end{align*}
By adding the two equations above and using the fact that the
characteristic of $\field$ is 0, we get $L(\unit_1)\MinkProd
L(\unit_1)=-\left(L(\unit_i)\MinkProd L(\unit_i)\right)$.  By
subtracting one equation from the other, we get $L(\unit_1)\MinkProd
L(\unit_i)=0$. Thus \eqref{e1-bizonyitando} holds.

If $d=2$, this implies by Lemma~\ref{PoiSim-lemma} that $L\in\PoiSim$, and we are done.
To prove that 
$L\in\PoiSim$ for $d>2$, it remains to prove that
\begin{equation}
\label{e2-bizonyitando}
L(\unit_i)\MinkProd L(\unit_j)=0\ \text{ if $1\neq i\neq j\neq 1$},
\end{equation} 
so choose any distinct $i, j \in  \{2,\dots,d\}$.

We observe that the set of rational numbers is contained
in $\F$ as the characteristic of $\field$ is $0$; and it is easy to check that 
$
\left(\frac{3}{5}\unit_i+\frac{4}{5}\unit_j\right)\,\llcon\,\unit_1.
$ 
It follows, using the fact that $L$ is both an automorphism of $\relst$ and
a linear transformation, that
$\left(\frac{3}{5}L(\unit_i)+\frac{4}{5}L(\unit_j)\right)\,\llcon\,
L(\unit_1)$. 
From this, together with \eqref{e-llcon}, \eqref{e1-bizonyitando}
and the fact that  $\MinkProd$ is a symmetric bilinear function, we calculate 
directly that                          
\begin{align*}
0
&=
\left[ \tfrac{3}{5}L(\unit_i)+\tfrac{4}{5}L(\unit_j)-L(\unit_1) \right]
\MinkProd
\left[ \tfrac{3}{5}L(\unit_i)+\tfrac{4}{5}L(\unit_j)-L(\unit_1) \right] 
\\
&=
\tfrac{ 9}{25} \left[ L(\unit_i)\MinkProd L(\unit_i) \right] +
\tfrac{16}{25} \left[ L(\unit_j)\MinkProd L(\unit_j) \right] +
               \left[ L(\unit_1)\MinkProd L(\unit_1) \right] +
\tfrac{24}{25} \left[ L(\unit_i)\MinkProd L(\unit_j) \right]
\\
&=
\tfrac{- 9}{25} \left[ L(\unit_1)\MinkProd L(\unit_1) \right] 
-\tfrac{16}{25} \left[ L(\unit_1)\MinkProd L(\unit_1) \right] 
+               \left[ L(\unit_1)\MinkProd L(\unit_1) \right] 
+\tfrac{24}{25} \left[ L(\unit_i)\MinkProd L(\unit_j) \right] 
\\
&=
 \tfrac{( -9 - 16 + 25)}{25} \left[ L(\unit_1)\MinkProd L(\unit_1) \right] 
+\tfrac{       24      }{25} \left[ L(\unit_i)\MinkProd L(\unit_j) \right] 
\\
&= 
\tfrac{24}{25}  \left( L(\unit_i)\MinkProd L(\unit_j) \right).
\end{align*}

Thus \eqref{e2-bizonyitando} holds, and we are done. 
\end{proofofpart}

\begin{proofofpart}{\ref{autgroups}\eqref{Beq4} $(\AffAut{\galst}=\GenGal)$}

First, we show that $\GenGal \subseteq \AffAut{\galst}$.  

We call  segment $(\vec{p},\vec{q}\,)$ 
spatial if{}f $\vec{p}\,\abstime\,\vec{q}$.
 It is easy to see from the definitions that
Galilean similarities take spatial
segments to spatial segments and
spatially
congruent segments to spatially congruent
ones. As the spatial square-factors are
nonzero, these similarities also take spatially noncongruent  
spatial segments to spatially noncongruent
ones.  
The inverse images of spatial segments
under Galilean similarities are also spatial
by Lemma~\ref{nevetkelladni-prop}\eqref{Ai2}.
Therefore, the inverse images of
spatially congruent segments have to be
spatially congruent. 
 Therefore, Galilean similarities respect
congruence on simultaneity $\cng{\abstime}$. They
also respect $\Bw$ as they are affine transformations.
Hence $\GalSim\subseteq \AffAut{\galst}$.

To
prove the converse, \ie that
$\AffAut{\galst}\subseteq\GalSim$, choose any 
$A \in \AffAut{\galst}$, and
write $A=\tau\circ L$, where $L$ is a linear
transformation and $\tau$ a translation.
We want to prove that
$A\in\GalSim$, but since 
$\tau\in\GalSim$ and $\GalSim$ is
closed under composition (this is easily checked), it is enough to prove
that $L\in\GalSim$.

Because $\tau^{-1} \in \AffAut{\galst}$, 
it is clear that 
$L = \tau^{-1} \circ A \in \AffAut{\galst}$ as well.
So $L$ respects $\cng{\abstime}$, and hence also
$\abstime$, because $\abstime$ can
be defined in terms of $\cng{\abstime}$
as $\vec{p}\,\abstime\,\vec{q}\; \Leftrightarrow\;
(\vec{p},\vec{q}\,)\,\cng{\abstime}\,(\vec{p},\vec{q}\,)$. 
Therefore,
$L$ is a bijection on $\simult$, \ie 
$L[\simult]=\simult$.
Moreover, it is easy to see
that $(\unit_i,\unit_j) \cng{\abstime} (\unit_i, -\unit_j)$ whenever $i > 1, j > 1$ are distinct.
As in the proof above for $\euclg$ 
(part (\ref{Beq1})) it follows that 
$L(\unit_i) \ScalProd L(\unit_j) = 0$ 
whenever $i, j \in\{2,\dots,d\}$ are distinct, 
and that there 
is some $a > 0$ for which 
$L(\unit_i) \ScalProd L(\unit_i) = a$ for all 
$i \in	\{2,\dots,d\}$. 
Hence, if $\vec{p}, \vec{q} \in F^d$ 
both have 
$\unit_1$-component 0, then 
$L(\vec{p}\,) \ScalProd L(\vec{q}\,) = a\cdot 
(\vec{p} \ScalProd \vec{q}\,)$.

To complete the proof that $L \in \GenGal$, 
let us suppose that
$\vec{p}\, \abstime\, \vec{q}$.
As $L$ respects $\abstime$,
we have
$L(\vec{p}\,)\, \abstime\, L(\vec{q}\,)$. 
Furthermore, the $\unit_1$-component of 
$(\vec{p}-\vec{q}\,)$ is 0 because 
$\vec{p}\,\abstime\,\vec{q}$, 
and as we have just seen, this
implies that $L(\vec{p}-\vec{q}\,) \ScalProd 
L(\vec{p}-\vec{q}\,) = a\cdot\left(( \vec{p} -
\vec{q}\, ) \ScalProd ( \vec{p} - \vec{q}\,)\right)$.
The claim follows.  
\end{proofofpart}

\begin{proofofpart}{\ref{autgroups}\eqref{Beq5} 
$(\AffAut{\newtst}=\GenTriv)$}
By definition, $\newtst \defeq \laa \galst,\rest\raa$
By \eqref{Beq4}, $\AffAut{\galst}=\GenGal$. 
Therefore, $\AffAut{\newtst}=\AffAut{\galst}\cap
\aut{\la\F^d,\rest\ra}=
\GenGal\cap\aut{\la F^d,\rest\ra} {\defeq}\GenTriv$.
\end{proofofpart}

\begin{proofofpart}{\ref{autgroups}\eqref{Beq6} 
$(\AffAut{\Lclst}=\TrivSim)$}
Since  $\Lclst \defeq \la\galst,\llcon\ra$
and  $\relst \defeq \la\F^d,\llcon,\Bw\ra$, 
we have that 
$\AffAut{\Lclst}=\AffAut{\relst}\cap\AffAut{\galst}$.
By \eqref{Beq3} and \eqref{Beq4},
$\AffAut{\relst}=\PoiSim$ and
$\AffAut{\galst}=\GenGal$.
Item \eqref{Ai6} of Lemma~\ref{nevetkelladni-prop}
says that $\PoiSim\cap\GenGal=\TrivSim$.
Therefore, $\AffAut{\Lclst}=\TrivSim$.
\end{proofofpart}

This completes the proofs establishing Theorem~\ref{autgroups}.

\subsection{Proof of Theorems~\ref{leiras} and 
\ref{leiras2}}

In the proofs 
of Theorems~\ref{leiras}
and \ref{leiras2}, we will use Theorem~\ref{re-thm-erlangen}\eqref{Er2} and Corollary~\ref{re-cor-erlangen}\eqref{ErCor2},
 which say that
\begin{align}
\label{ER2}
\tag{Theorem~\ref{re-thm-erlangen}\eqref{Er2}}
\ca{\geometry}\subseteq\ca{\geometry'} & \ 
\Longleftrightarrow  \
 \AffAut{\geometry}\supseteq
\AffAut{\geometry'},\\
\label{CER2}
\ca{\geometry}=\ca{\geometry'} & \ 
\Longleftrightarrow  \
\AffAut{\geometry}=\AffAut{\geometry'}
\tag{Corollary~\ref{re-cor-erlangen}\eqref{ErCor2}}.
\end{align}

\subsubsection{Proof of Theorem \ref{leiras}}

\begin{proofofpart}{\ref{leiras}\eqref{Le1} $(\ca{\relst}=\ca{\minkst})$}
Follows from \ref{CER2}  and 
Theorem~\ref{autgroups}\eqref{Beq3}, which
says that $\AffAut{\relst}=\AffAut{\minkst}$.
\end{proofofpart}

\begin{proofofpart}{\ref{leiras}\eqref{Le2} $(\ca{\galst}\subseteq\ca{\newtst})$}
Because $\newtst \defeq \la\galst,\rest\ra$.
\end{proofofpart}

\begin{proofofpart}{\ref{leiras}\eqref{Le3} $(\ca{\OAff}\subseteq 
\left(\ca{\euclg}\cap\ca{\relst}\cap\ca{\galst}\right))$}
Because $\Bw$ is a concept of  $\euclg$, $\relst$ and $\galst$.
\end{proofofpart}

\begin{proofofpart}{\ref{leiras}\eqref{Le4} 
$\left(\left(\ca{\euclg}\cup\ca{\relst}\cup
\ca{\newtst}\right)\subseteq
\ca{\Lclst}\right)$}
By Theorem~\ref{autgroups}, we
have that 
$\AffAut{\euclg}=\EuclSim$,
$\AffAut{\relst}=\PoiSim$,
$\AffAut{\newtst}= \GenTriv$ and 
$\AffAut{\Lclst}=\TrivSim$.
By 
Lemma~\ref{nevetkelladni-prop}, 
each of
$\EuclSim$, $\PoiSim$ and
$\GenTriv$ contains  $\TrivSim$, 
whence each of
$\AffAut{\euclg}$,
 $\AffAut{\relst}$
and $\AffAut{\newtst}$
contains  $\AffAut{\Lclst}$. 
This implies, by \ref{ER2} that
$\ca{\euclg}$, $\ca{\relst}$ and
$\ca{\newtst}$ are all contained
in $\ca{\Lclst}$, as claimed.
\end{proofofpart}

\begin{proofofpart}{\ref{leiras}\eqref{Le5} $($correctness of Table~\ref{tab:concept-locations}$)$}

We have already shown (\ref{leiras}\eqref{Le1}) that 
$  \ca{\la \F^d,\llcon,\Bw\ra} 
\defeq  \ca \relst
=  \ca \minkst
\defeq  \ca{\la\F^d,\cng{\mu},\Bw\ra}$. 
Therefore, for any coordinate geometry
$\geometry$ over $\field$, we have that
\begin{equation}
\label{interdef1}
\llcon\in\ca{\geometry}\quad\Longleftrightarrow\quad 
\cng{\mu}\in\ca{\geometry}.
\end{equation}
Similarly, by \eqref{deltaeq}, we have 
\begin{equation}
\label{interdef2}
\bm{\updelta}\in\ca{\geometry}\quad \Longleftrightarrow\quad 
\{ \llcon,\abstime \} \subseteq \ca{\geometry}.
\end{equation}
Furthermore, since $\abstime$ is definable in terms
of $\cng{\abstime}$ 
(as $\vec{p}\,\abstime\, \vec{q}\; \Leftrightarrow\;
(\vec{p},\vec{q}\,)\cng{\abstime}(\vec{p},\vec{q}\,)$), we have
that
\begin{equation}
\label{interdef3}
\left(\cng{\abstime}\in\ca{\geometry}\ \Longrightarrow\ \abstime\in
\ca{\geometry}\right)\
\text{ and }\ \left(\abstime\not\in\ca{\geometry}\ \Longrightarrow\
\cng{\abstime}\not\in\ca{\geometry}\right).
\end{equation} 
We now prove the correctness of the Table, one column at a time. For techniocal reasons, we deal with column 4 ($\newtst$)
before column 3 ($\galst$).

\medskip

\underline{Table~\ref{tab:concept-locations} column 1: $\euclg \defeq \la \F^d,\cng{},\Bw\ra$}.
\emph{Required to prove}: 
$\cng{}$ is a concept; and
$\cng{\mu}$, $\llcon$, $\cng{\abstime}$, $\abstime$,
$\rest$ and $\bm{\updelta}$ are not concepts, of 
$\euclg$.

Clearly, $\cng{}\in\ca{\euclg}$.
 To prove that the other
relations are not concepts, it is enough to
find an  automorphism of $\euclg$ that does 
not respect
them.  Let $E\colon\F^d\rightarrow \F^d$ 
be the linear transformation that 
takes the unit vectors $\unit_1$ and
$\unit_2$ respectively to
$\frac{3}{5}\unit_1+\frac{4}{5}\unit_2$
and $\frac{4}{5}\unit_1-\frac{3}{5}\unit_2$, and
keeps the rest of the unit vectors fixed.
Of course $E(\origin\,)=\origin$ as
$E$ is a linear transformation. Transformation
$E$ is therefore a bijection as it takes the
unit vectors to $d$ distinct non-zero vectors, 
which are linearly independent 
and therefore form a basis for $F^d$.
By Lemma~\ref{EuclSim-lemma}\eqref{EuclSim-units}, 
it is easy to 
check that $E$ is an  Euclidean similarity, 
hence by Theorem~\ref{autgroups}\eqref{Beq1},
$E$ is an automorphism of $\euclg$.
Transformation $E$ respects neither of $\rest$,
$\abstime$, $\llcon$ 
 because
$\origin\,\rest\,\unit_1$,
$\origin\,\abstime\,\unit_2$ and
$\origin\,\llcon\, (\unit_1+\unit_2)$ all hold, but
none of
$E(\origin\,)\,\rest\, E(\unit_1)$,
$E(\origin\,)\,\abstime\, E(\unit_2)$,
$E(\origin\,)\,\llcon\, E(\unit_1+\unit_2)$
do. 
Therefore, $\rest,\abstime,\llcon\not\in
\ca{\euclg}$. By these, \eqref{interdef1},
\eqref{interdef2} and \eqref{interdef3}, we have
that $\cng{\mu}, \cng{\abstime},\bm{\updelta}
\not\in\ca{\euclg}$.

{\medskip}

\underline{Table~\ref{tab:concept-locations} column 2: $\relst \defeq \la\F^d,\llcon,\Bw\ra$}.
\emph{Required to prove}: $\cng{\mu}$ and $\llcon$ are concepts; and
$\cng{}$, $\cng{\abstime}$, $\abstime$,
$\rest$ and $\bm{\updelta}$
are not concepts, of $\relst$.

Clearly, $\llcon\in\ca{\relst}$.
Therefore, by \eqref{interdef1}, 
$\cng{\mu}\in\ca{\relst}$.
To prove that the remaining relations are
not concepts, we will find an automorphism
of $\relst$ that does not respect them. 
Let $P\colon\F^d\rightarrow\F^d$ be the
linear transformation that takes the
unit vectors $\unit_1$ and $\unit_2$ respectively to
 $\frac{5}{3}\unit_1+\frac{4}{3}\unit_2$ 
and $\frac{4}{3}\unit_1+\frac{5}{3}\unit_2$
and keeps the rest of the unit vectors fixed.
Transformation $P$ is a bijection as
it takes the unit vectors to linearly 
independent vectors.
By Lemma~\ref{PoiSim-lemma}, it is easy to
see that $P$ is a  Poincar\'e similarity.
Therefore, by Theorem~\ref{autgroups}{\eqref{Beq3}},
$P$ is an automorphism of $\relst$. 
However, 
$P$ respects neither $\cng{}$, $\abstime$ nor
$\rest$, because $\origin\, \rest\, \unit_1$,
$\origin\,\abstime\,\unit_2$ and 
$(\origin,\unit_1-\unit_2)\cng{}
(\origin,\unit_1+\unit_2)$
all hold, but none of   
$P(\origin\,)\, \rest\, P(\unit_1)$,
$P(\origin\,)\,\abstime\,P(\unit_2)$, 
$(P(\origin\,),P(\unit_1-\unit_2))\cng{}
(P(\origin\,),P(\unit_1+\unit_2))$ do.
Therefore, 
$\cng{},\abstime,\rest\not\in\ca{\relst}$.
By \eqref{interdef2} and \eqref{interdef3}, 
$\cng{\abstime},\bm{\updelta}\not\in\ca{\relst}$.

\medskip

\underline{Table~\ref{tab:concept-locations} column 4: $\newtst \defeq \la \F^d,\cng{\abstime},\rest,\Bw\ra$}.
\emph{Required to prove}: $\cng{\abstime}$,
$\abstime$ and $\rest$ are concepts; and
$\cng{}$, $\cng{\mu}$, $\llcon$ and  $\bm{\updelta}$
are not concepts, of $\newtst$.

By definition of $\newtst$ and
\eqref{interdef3},
we have that $\{ \cng{\abstime},\abstime,
\rest\} \subseteq \ca{\newtst}$. 
To prove that the remaining relations are not
concepts, it is enough to find an automorphism
of $\newtst$ that does not respect them.
Let $N$ be the linear 
transformation that takes  
$\unit_1$ to $2\unit_1$ and keeps the rest of 
the
unit vectors fixed. 
It is easy to see that $N$ is a bijection that
respects $\rest$ and $\abstime$, and
that it ``does not change the squared Euclidean
distance''
between $\abstime$-related points,
\ie
$\sqEd{\vec{p}}{\vec{q}}=\sqEd{N(\vec{p}\,)}{N(\vec{q}\,)}$
if $\vec{p}\,\abstime\,\vec{q}$. Therefore,
$N$ also respects $\cng{\abstime}$, and
is an automorphism of $\newtst$.
However, $N$
respects neither $\cng{}$ nor $\llcon$ because
$(\origin,\unit_1)\,\cng{}\,(\origin,\unit_2)$ 
and $\origin\,\llcon\,(\unit_1+\unit_2)$ hold, but
neither $\left(N(\origin\,),N(\unit_1)\right)\,
\cng{}\, \left(N(\origin\,),N(\unit_2)\right)$
nor $N(\origin\,)\,\llcon\, N(\unit_1+\unit_2)$
hold. Therefore, 
$\cng{},\llcon\not\in\ca{\newtst}$. By 
\eqref{interdef1} and \eqref{interdef2},
we have that $\cng{\mu},\bm{\updelta}\not\in\ca{\newtst}$.

\medskip

\underline{Table~\ref{tab:concept-locations} column 3: $\galst \defeq \la\F^d,\cng{\abstime},\Bw\ra$}.
\emph{Required to prove}: $\cng{\abstime}$,
$\abstime$ are concepts; and $\cng{}$, $\cng{\mu}$,
$\llcon$, $\rest$ and $\bm{\updelta}$ are not
concepts of $\galst$.

We have 
that
$\{ \cng{\abstime},\abstime \} \subseteq \ca{\galst}$ by 
definition of $\galst$ and \eqref{interdef3}.
 By $\ca{\galst}\subseteq\ca{\newtst}$ and
$\cng{},\cng{\mu},\llcon,\bm{\updelta}\not
\in\ca{\newtst}$, we have that
 $\cng{},\cng{\mu},\llcon,\bm{\updelta}\not
\in\ca{\galst}$.
To prove that  
$\rest\not\in\ca{\galst}$, we will find an
automorphism of $\galst$ that does not respect $\rest$. 
Let
$G\colon\F^d\rightarrow\F^d$ be the linear 
transformation that takes 
$\unit_1$ to $\unit_1+\unit_2$ and keeps the
 other unit vectors fixed. It is easy to
see that $G$ is a bijection that takes 
$\abstime$-related points to 
$\abstime$-related points
and it ``does not change the squared Euclidean
distance''
between $\abstime$-related points. Therefore,
$G$ is an automorphism of $\galst$. But clearly,
 $G$ does not respect $\rest$,
because $\origin\,\rest\,\unit_1$ holds,
but $G(\origin\,)\,\rest\, G(\unit_1)$ does 
not hold. Therefore, $\rest\not\in\ca{\galst}$.

\medskip

\underline{Table~\ref{tab:concept-locations} column 5: $\Lclst \defeq \la F^d, \cng{\abstime}, \llcon, \Bw \ra$}.
\emph{Required to prove}: $\cng{}$,  $\cng{\mu}$, $\llcon$, $\cng{\abstime}$,
$\abstime$, $\rest$ and $\bm{\updelta}$ are all concepts of
$\Lclst$.

We have already seen that
$\cng{}\in\ca{\euclg}$; 
$\{ \cng{\mu},\llcon \} \subseteq
\ca{\relst}$; and 
$\{ \cng{\abstime},\abstime,\rest \} \subseteq \ca{\newtst}$;
since Theorem \ref{leiras}\eqref{Le3} tells us that
\[
\left(\ca{\euclg}\cup\ca{\relst}\cup
\ca{\newtst}\right)\subseteq\ca{\Lclst},
\]
we have $\{ \cng{},\cng{\mu},\llcon,\cng{\abstime},
\abstime,\rest \} \subseteq \ca{\Lclst}$. Finally, by \eqref{interdef2}, we also have that $\bm{\updelta}\in\ca{\Lclst}$.
\end{proofofpart}

This completes the proofs required for Theorem \ref{leiras}.

\subsubsection{Proof of Theorem \ref{leiras2}}

\begin{proofofpart}{\ref{leiras2}\eqref{LL1} and \ref{leiras2}\eqref{LL2}}

Since the proofs of parts \eqref{LL1} and \eqref{LL2} overlap to some extent, we will prove them together,
considering each possible choice of $\geometry$ in turn. But first, we make some general observations that will 
simplify things below.

Assume that $\geometry$ is  a coordinate geometry
such that $\ca{\geometry}\subseteq\ca{\Lclst}$,
and that relation $\rel$ is definable
in terms of relation $\rel^+$ where  
$\{ \rel,\rel^+ \} \subseteq \ca{\Lclst}$. Then
\[
\ca{\la\geometry,\rel\ra}=
\ca{\Lclst}\quad \Longrightarrow\quad 
\ca{\la\geometry,\rel^+\ra}=\ca{\Lclst}
\]
holds 
because $\ca{\la \geometry,\rel\ra}
\subseteq\ca{\la \geometry,\rel^+\ra}
\subseteq\ca{\Lclst}$.
By Theorem~\ref{leiras}, for any
$\geometry\in\{\relst, \euclg, \galst, 
\newtst\}$ and  
$\rel\in\{\cng{},\cng{\mu},\llcon, \cng{\abstime}, \abstime,
\rest, \bm{\updelta}\}$, we have that
$\ca{\geometry}\subseteq\ca{\Lclst}$ and
$\rel\in\ca{\Lclst}$. Therefore,
for any geometry 
$\geometry\in\{\relst, \euclg, \galst, 
\newtst\}$, we have that
\begin{align}
\label{Or1}
\ca{\la \geometry,\abstime\ra}=\ca{\Lclst}
\quad \Longrightarrow & \quad \ca{\la 
\geometry,\cng{\abstime}\ra}=\ca{\Lclst}, \\
\label{Or2}
\ca{\la \geometry,\abstime\ra}=\ca{\Lclst}
\quad \Longrightarrow & \quad \ca{\la 
\geometry,\bm{\updelta}\ra}=\ca{\Lclst},\\ 
\label{Or3}
\ca{\la \geometry,\llcon\ra}=\ca{\Lclst}
\quad \Longrightarrow & \quad \ca{\la 
\geometry,\cng{\mu}\ra}=\ca{\Lclst}, \\
\label{Or4}
\ca{\la \geometry,\llcon\ra}=\ca{\Lclst}
\quad \Longrightarrow & \quad \ca{\la 
\geometry,\bm{\updelta}\ra}=\ca{\Lclst} 
\end{align}
because $\abstime$ and $\lambda$ are definable
in  terms of $\bm{\updelta}$ by \eqref{deltaeq},
$\abstime$ is definable in terms of 
$\cng{\abstime}$ (as $\vec{p}\,\abstime\,\vec{q}
\,\Leftrightarrow\, (\vec{p},\vec{q}\,)\cng{\abstime}(\vec{p},\vec{q}\,)$)
 and $\llcon$ is definable
in terms of $\cng{\mu}$ 
(as $\vec{p}\,\llcon\,\vec{q}\,\Leftrightarrow\,
(\vec{p},\vec{q}\,)\cng{\mu}(\vec{p},\vec{p}\,)$).

Also, by Theorem~\ref{autgroups}\eqref{Beq6} (which
says that $\AffAut{\Lclst}=\TrivSim$) and
by \ref{CER2}, we have for any finitely field-definable coordinate geometry $\geometry$, that
\begin{equation}
\label{Mik}
\ca{\la \geometry,\rel\ra}=\ca{\Lclst}\
\ \Longleftrightarrow\ \AffAut{\la \geometry,
\rel\ra}=\TrivSim.
\end{equation}

\medskip

\underline{Case $\geometry = \Eucl$}.
\emph{Required to prove}: 
$\la \euclg, \cng{\abstime}\ra$,
 $\la \euclg, \abstime\ra$, $\la \euclg, \rest\ra$
and $\la \euclg,\bm{\updelta}\ra$ are definitionally
equivalent to $\Lclst$; and if $d>2$,
$\la \euclg,\cng{\mu}\ra$ and $\la\euclg,\llcon\ra$ 
are definitionally equivalent to
$\Lclst$.

By 
Theorem~\ref{autgroups}\eqref{Beq1},
$\AffAut{\euclg}=\EuclSim$.
By this and  
Lemma~\ref{nevetkelladni-prop}\eqref{Ai4}, 
we have
that 
\begin{equation}
\label{EuS}
\begin{aligned}
\AffAut{\la \euclg,\abstime\ra} 
& = 
\AffAut{\euclg}\cap
\aut{\la \F^d,\abstime\ra} \\ 
& =
\EuclSim\cap 
\aut{\la\F^d,\abstime\ra}
=
\TrivSim
\end{aligned}
\end{equation}
and likewise
\begin{equation}
\label{EuR}
\AffAut{\la \euclg,\rest\ra}=\EuclSim\cap\aut{
\la\F^d,\rest\ra}=\TrivSim.
\end{equation}
Similarly, by 
Lemma~\ref{nevetkelladni-prop}\eqref{Ai5}, by
$\relst \defeq \la\F^d,\llcon,\Bw\ra$ and by $\AffAut{\relst}=
\PoiSim$ (Theorem \ref{autgroups}{\eqref{Beq3}}), we have, when $d > 2$,
\begin{equation}
\label{EuL}
\begin{aligned}
\AffAut{\la \euclg,\llcon\ra}
&=
\AffAut{\euclg}\cap\AffAut{\relst} \\
& =
\EuclSim\cap\PoiSim 
 =
\TrivSim
\end{aligned}
\end{equation}
Therefore, by \eqref{Mik}, \eqref{EuS},
\eqref{EuR}, \eqref{Or1} and \eqref{Or2},
we have that 
$\ca{\la\euclg,\rest\ra}=\ca{\la \euclg,
\abstime\ra}=\ca{\la \euclg,\cng{\abstime}\ra}=
\ca{\la \euclg,\bm{\updelta}\ra}=\ca{\Lclst}$.
And by \eqref{Mik}, \eqref{EuL} 
and
\eqref{Or3}, we also have 
$\ca{\la\euclg,\llcon\ra}=
\ca{\la\euclg,\cng{\mu}\ra}=\ca{\Lclst}$
when $d>2$.

{\medskip}

\underline{Case $\geometry =\relst$}:
\emph{Required to prove}: 
$\la\relst,\cng{\abstime}\ra$,
$\la \relst,\abstime\ra$, $\la\relst,\rest\ra$ and
$\la \relst, \bm{\updelta}\ra$ are definitionally
equivalent to $\Lclst$; and if $d>2$,
$\la\relst,\cng{}\ra$ is definitionally equivalent
to $\Lclst$.

By Theorem~\ref{autgroups}{\eqref{Beq3}}, we have that
$\AffAut{\relst}=\PoiSim$. By this and
Lemma~\ref{nevetkelladni-prop}\eqref{Ai3},
\begin{alignat}{4}
\label{ReS}
\AffAut{\la\relst,\abstime\ra} & = \PoiSim\cap\aut{\la \F^d,\abstime\ra} &~= \TrivSim, & \text{ and }  \\
\label{ReR}
\AffAut{\la\relst,\rest\ra}    & = \PoiSim\cap\aut{\la \F^d,\rest\ra}    &~= \TrivSim.          
\end{alignat}
By \eqref{Mik}, \eqref{ReS}, \eqref{ReR}, \eqref{Or1}
and \eqref{Or2}, we have that
$\ca{\la\relst,\rest\ra}=
\ca{\la\relst,\abstime\ra}=
\ca{\la\relst,\cng{\abstime}\ra}=
\ca{\la\relst,\bm{\updelta}\ra}=\ca{\Lclst}$.

From the definitions of $\euclg$ and $\relst$,
we have that $\la \relst,\cng{}\ra \doteq
\la\euclg,\llcon\ra$.
We have already seen that
$\ca{\la\euclg,\llcon\ra}=\ca{\Lclst}$ if
$d>2$. Thus 
$\ca{\la\relst,\cng{}\ra}=\ca{\Lclst}$
if $d>2$.

{\medskip}

\underline{Case $\geometry = \galst$}.
\emph{Required to prove}:
$\la \galst,\cng{}\ra$,
$\la\galst,\cng{\mu}\ra$,
$\la\galst,\llcon\ra$ and
$\la\galst,\bm{\updelta}\ra$
are all definitionally equivalent
to $\Lclst$.

Recall that $\Lclst \defeq \la \galst,\llcon\ra$.
So by \eqref{Or3} and \eqref{Or4},
we have that $\ca{\la\galst,\cng{\mu}\ra}=
\ca{\la\galst,\bm{\updelta}\ra}
=\ca{\Lclst}$.

Now recall that $\galst \defeq \la\F^d,\cng{\abstime},
\Bw\ra$ and
$\euclg \defeq \la\F^d,\cng{},\Bw\ra$;
hence $\la \galst,\cng{}\ra 
\doteq \la\euclg,\cng{\abstime}\ra$.
We have already seen that
$\ca{\la\euclg,\cng{\abstime}\ra}=
\ca{\Lclst}$. Therefore, 
$\ca{\la \galst,\cng{}\ra}=\ca{\Lclst}$.

{\medskip}

\underline{Case $\geometry = \newtst$}.
\emph{Required to prove}:
$\la\newtst,\cng{}\ra$,
$\la\newtst,\cng{\mu}\ra$,
$\la\newtst,\llcon\ra$ and
$\la\newtst,\bm{\updelta}\ra$
are all definitionally equivalent to $\Lclst$.

Recall from Theorem \ref{leiras}(\ref{Le2},\ref{Le4},\ref{Le5}) 
that $\ca{\galst}\subseteq\ca{\newtst}\subseteq
\Lclst$
and $\{\cng{},\cng{\mu},\llcon,\bm{\updelta}\}
\subseteq\Lclst$. It follows, for any $\rel\in\{\cng{},\cng{\mu},
\llcon,\bm{\updelta}\}$, that
\[
\ca{\la \galst,\rel\ra}\subseteq\ca{\la \newtst,\rel\ra}
\subseteq\ca{\Lclst}.
\]
We have already seen that for any
$\rel\in\{\cng{},\cng{\mu},
\llcon,\bm{\updelta}\}$,
$\ca{\la\galst,\rel\ra}=\ca{\Lclst}$.
Therefore,
for any $\rel\in\{\cng{},\cng{\mu},
\llcon,\bm{\updelta}\}$,
$\ca{\la\newtst,\rel\ra}=\ca{\Lclst}$.

\medskip
This completes the proofs of items \ref{leiras2}\eqref{LL1}
and \ref{leiras2}\eqref{LL2}.
\end{proofofpart}

\begin{proofofpart}{\ref{leiras2}\eqref{LL3}}
We have to show that 
  $\la \relst, \cng{} \ra$,
	$\la \euclg, \cng{\mu} \ra$ and
	$\la \euclg, \llcon \ra$
are definitionally equivalent to one another; and that they are not 
definitionally equivalent to $\Lclst$ when $d=2$.

Note first that $\la \euclg,\llcon\ra$ and $\la\relst,\cng{}\ra$
are clearly definitionally equivalent, because $\euclg \defeq \la \F^d,\cng{},\Bw\ra$ and
$\relst \defeq \la \F^d,\llcon,\Bw\ra$.
Geometries $\la \euclg,\cng{\mu}\ra$ and
$\la\euclg,\llcon\ra$ are definitionally
equivalent, too,  because we already know from Theorem~\ref{leiras}\eqref{Le1} that
$\minkst \defeq \la\F^d,\cng{\mu},\Bw\ra$ and $\relst \defeq \la \F^d,\llcon,\Bw\ra$
 are definitionally
equivalent.
However, geometries $\la \euclg,\llcon\ra$
and $\Lclst$ are not definitionally equivalent if $d=2$,
because the linear transformation that interchanges
the unit vectors $\unit_1$ and $\unit_2$ is
an automorphism of $\la \euclg,\llcon\ra$, but it
is not an automorphism of $\Lclst$ as it does not respect
$\cng{\abstime}$. 
Finally, it follows from transitivity of  
that 
$\la \relst, \cng{} \ra$ and 	$\la \euclg, \cng{\mu} \ra$
are not definitionally equivalent to $\lclst$ either, when $d=2$.
\end{proofofpart}

\begin{proofofpart}{\ref{leiras2}\eqref{LL4}}
We have to show that $\la \OAff,\bm{\updelta}\ra$ is 
definitionally equal to $\Lclst$.

Recall that $\OAff \defeq \la\F^d,\Bw\ra$ and
$\relst \defeq \la\F^d,\llcon,\Bw\ra$,
and that  
$\ca{\la\F^d,\bm{\updelta}\ra}=
\ca{\la\F^d,\llcon,\abstime\ra}$
by \eqref{deltaeq}.
We have already seen that
$\ca{\la\relst,\abstime\ra}=
\ca{\Lclst}$.
Therefore, $\ca{\la \OAff,\bm{\updelta}\ra}=
\ca{\la \OAff,\llcon,\abstime\ra}=
\ca{\la \relst,\abstime\ra}=\ca{\Lclst}$.
\end{proofofpart}

This completes the proofs required for Theorem \ref{leiras2}.

\section{Concluding remarks}

In this paper, we have shown the connection between the concept-sets
of Galilean, Newtonian, Late Classical, Relativistic and Minkowski
spacetimes and ordered affine and Euclidean geometries illustrated by
Figures~\ref{fig:hasse} and \ref{tartalmazas-fig}.  Here are some open
problems:
\begin{problem}
  Is there a common concept of Newtonian spacetime and Relativistic spacetime which is not a concept of ordered affine geometry? In other words, is  $\ca{\OAff}$ a proper subset of the intersection $\ca{\newtst} \cap\ca{\relst}$?  
\end{problem}
\begin{problem}
  Is there a common concept of Galilean spacetime and Relativistic spacetime which is not a concept of ordered affine geometry? In other words, is  $\ca{\OAff}$ a proper subset of the intersection $\ca{\galst} \cap\ca{\relst}$?  
\end{problem}
\begin{problem}
Is there a common concept of Euclidean geometry and Galilean spacetime which is not a concept of ordered affine geometry? In other words, is  $\ca{\OAff}$ a proper subset of the intersection $\ca{\euclg} \cap\ca{\galst}$?  
\end{problem}
\begin{problem}
  Is there a common concept of Euclidean geometry and Newtonian spacetime which is not a concept of ordered affine geometry? In other words, is  $\ca{\OAff}$ a proper subset of the intersection $\ca{\euclg} \cap\ca{\newtst}$?  
\end{problem}
\begin{problem}
  Is there a common concept of Euclidean geometry and Relativistic spacetime which is not a concept of ordered affine geometry? In other words, is  $\ca{\OAff}$ a proper subset of the intersection $\ca{\euclg} \cap\ca{\relst}$?  
\end{problem}

Since $\ca{\galst}\subseteq \ca{\newtst}$, Problems 1 and 2
are not independent, and the same holds for Problems 3 and 4.

\bibliographystyle{amsalpha}
\bibliography{LogRel12019}

\end{document}